\documentclass{amsart}[12pt]

\usepackage{amsfonts,amssymb,amsmath,amsthm}
\usepackage{amsaddr}

\usepackage{graphicx}
\usepackage{epstopdf}
\usepackage{caption}
\usepackage{subcaption}
\usepackage{color}
\usepackage{marginnote}

\usepackage[pdfpagelabels]{hyperref}

\newtheorem{thm}{Theorem}[section]
\newtheorem{defi}{Definition}[section]
\newtheorem{prop}{Proposition}[section]
\newtheorem{lem}{Lemma}[section]
\newtheorem{cor}{Corollary}[section]
\newtheorem{rem}{Remark}[section]

\newcommand{\beg}{\begin}
\newcommand{\bea}{\beg{eqnarray}}
\newcommand{\eea}{\end{eqnarray}}

\def\lb{\label}

\newcommand{\rr}{\mathbb{R}}

\newcommand{\zz}{\mathbb{Z}}

\newcommand{\Sp}{\mathrm{Sp}}

\newcommand{\ga}{\gamma}
\newcommand{\gm}{\gamma}

\newcommand{\lmd}{\lambda}
\newcommand{\zt}{\zeta}
\newcommand{\sg}{\sigma}
\newcommand{\om}{\omega}

\newcommand{\bt}{\beta}

\newcommand{\vep}{\varepsilon}

\newcommand{\vr}{\varrho}

\newcommand{\ey}{\frac{1}{2}}

\newcommand{\xd}{\dot{x}}

\newcommand{\qd}{\dot{q}}

\newcommand{\bh}{\hat{B}}

\newcommand{\bc}{\mathcal{B}}

\newcommand{\X}{\mathcal{X}}
\newcommand{\Xh}{\hat{\mathcal{X}}}
\newcommand{\E}{\mathcal{E}}
\newcommand{\A}{\mathcal{A}}
\newcommand{\I}{\mathcal{I}}
\newcommand{\U}{\hat{U}}

\newcommand{\lan}{\langle}
\newcommand{\ran}{\rangle}

\newcommand{\Mh}{\hat{M}}

\newcommand{\vd}{V_{\mathfrak{D}}}
\newcommand{\vn}{V_{\mathfrak{N}}}

\begin{document}

\title[An Index Theory for $n$-body problem]{An Index Theory for Collision, Parabolic and Hyperbolic  Solutions of the Newtonian $n$-body Problem}

\author{Xijun Hu}
\address{School of Mathematics, Shandong University, Jinan, Shandong, People’s Republic of  China}
\email{xjhu@sdu.edu.cn}

\author{Yuwei Ou}
\address{School of Mathematics (Zhuhai), Sun Yat-Sen University, Zhuhai, Guangdong, People’s Republic of China}
\email{ouyw3@mail.sysu.edu.cn}

\author{Guowei Yu}
\address{Chern Institute of Mathematics and LPMC, Nankai University, Tianjin, People's Republic of China}
\email{yugw@nankai.edu.cn}

\thanks{This work was partially supported by National Key R\&D Program of China(2020YFA0713300).\\ The first author thanks the support of NSFC(No.12071255, 11790271). The second author thanks support of NSFC (No.11801583).  The last author thanks the support of MSRI at Berkeley (under the NSF Grant No. DMS-1440140) and Nankai Zhide Foundation.}

%\author{Xijun Hu$^{1}$\thanks{Partially supported by NSFC(No.11425105, No.11131004) and NCET, E-mail:xjhu@sdu.edu.cn }
%\quad Guowei Yu$^{2}$ \thanks{Partially supported by FSMP, PSL and NSFC(No.11425105) E-mail:yu@ceremade.dauphine.fr}

\date{}

\begin{abstract}
In the Newtonian $n$-body problem for solutions with arbitrary energy, which start and end either at a total collision or a parabolic/hyperbolic infinity, we prove some basic results about their Morse and Maslov indices. Moreover for homothetic solutions with arbitrary energy, we give a simple and precise formula that relates the Morse indices of these homothetic solutions to the spectra of the normalized potential at the corresponding central configurations. Potentially these results could be useful in the application of non-action minimization methods in the Newtonian $n$-body problem.
\end{abstract}

\maketitle

\bigskip


\bigskip


\section{Introduction} \label{sec: intro}

The Newtonian $n$-body problem studies the motion of $n$ point masses, $m_i>0$, according to Newton's law of universal gravitation. Let $M=\text{diag}{(m_1I_d,\dots,m_n I_d)}$ be the mass matrix, where $I_d$ is the $d \times d$ identity matrix with $d \ge 1$. Then $q=(q_i)_{i=1}^n$ ($q_i \in \rr^d$ represents the position of $m_i$) satisfies
\begin{equation}\label{eq:NewtonINTRO}
M\,\ddot q = \nabla U(q),
\end{equation}
where $U(q)= \sum_{1\leq i< j \leq n}\frac{m_i m_j}{|q_i-q_j|}$ is the potential function (the negative potential energy) and $\nabla$ is the gradient with respect to the Euclidean metric.

The solutions of \eqref{eq:NewtonINTRO} are invariant under linear translations, so there is no loss of generality to restrict ourselves to the $n^*:=d(n-1)$ dimension subspace
$$
\X:=\{ q \in \rr^{dn}: \; \sum_{i=1}^n m_i q_i =0 \},
$$
where the center of mass is fixed at the origin.

Let $T\X$ be the tangent bundle of $\X$. The Lagrangian $L: T\X \to [0, +\infty) \cup \{+\infty\}$
\begin{equation}\label{eq:lagrangianaINTRO}
 L(q,v)=K(v)+ U(q), \; \text{ where } K(v):=\frac{1}{2}|v|_M^2 := \ey \lan Mv, v \ran, \;
\end{equation}
has singularities at the collision configurations
$$ \Delta= \cup_{1 \le i < j \le n} \Delta_{ij}, \; \text{ where } \Delta_{ij}= \{q \in \X: \; q_i = q_j \}. $$
It is well-known (see \cite{AZ94}) the Lagrangian action functional
\begin{equation}
\mathcal{A}(q; t_1, t_2):=\int_{t_1}^{t_2} L(q(t), \dot{q}(t)) dt,
\end{equation}
is $C^2$ on $W^{1,2}([t_1, t_2], \Xh)$, $\Xh:= \X \setminus \Delta$ represents collision-free configurations, and any collision-free critical point of $\mathcal{A}$ is a classical solution of \eqref{eq:NewtonINTRO}.

Because the Newtonian gravity is a weak force, the action value of a path with collisions could still be finite. This means the critical points obtained using variational methods may contain a subset of collision moments with zero measure and only satisfies \eqref{eq:NewtonINTRO} in the complement of it. Such solutions were named \emph{generalized solutions} by Bahri and Rabinowitz in \cite{BR89} and \cite{BR91}.

In the last twenty years, a lot of progress has been made regarding how to rule out collisions in minimal critical points for the Newtonian $n$-body problem (or the general weak force $n$-body problem), details can be found in \cite{CM00}, \cite{C02}, \cite{FT04}, \cite{Ch08}, \cite{Y17} and the references within. These achievements are based on understanding of the asymptotic behaviors of the masses as they approach to a collision. To give a brief explanation, let's assume the solution $q(t)$ has a total collision at a moment $t_0$, i.e.,
\begin{equation*}
\lim_{t \to t_0^{\pm}} q_i(t) =0, \;\; \forall i = 1, \dots, n.
\end{equation*}
By introducing the polar coordinates:
$$ r= \sqrt{\mathcal{I}(q)},  \;\; s= (s_i)_{i=1}^n =q/r= (q_i/r)_{i=1}^n,$$
where $\I(q)= \lan Mq, q \ran$ is the \emph{moment of inertia}  and $\E:= \{q \in \X: \I(q)=1 \}$ is the set of \emph{normalized configurations}.
One finds that as $t \to t_0^{\pm}$, $r(t)$ satisfies the Sundman-Sperling estimate (see Lemma \ref{lem: asy1}) and $s(t)=q(t)/r(t)$ converges to the set of \emph{normalized central configurations}, as $t \to t_0^{\pm}$ (it is not clear whether there will be a definite limit, see \cite{C98}). A normalized central configuration $s \in \E$ is where the gradient of $U$ restricted on $\E$, $\nabla U|_{\E}(s)= \nabla U(s) + U(s)Ms$, vanishes.

Meanwhile very few results are available on how to rule out collision in non-minimal critical points. In \cite{Y17d}, based on an idea of Tanaka \cite{Tn93a}, it was shown the Morse index can be used to given an upper bound of the number of possible binary collisions (in certain cases eliminate all the possible binary collisions). We believe similar results should be available for collisions involving more than two masses. However this demands an efficient way of computing the Morse indices of collision solutions, and this will be one of the main results of our paper.

When a solution of the $n$-body problem does not experience any collision or non-collision singularity (see \cite{Xia}) in the future or the past, then a natural and important question is about the final motion of the masses as times goes positive or negative infinity. A classification of possible final motions were listed by Chazy in \cite{Chazy29}. In this paper we will focus on two of the simplest cases, the total parabolic/hyperbolic motion, which for simplicity will be referred as parabolic/hyperbolic motion.

\begin{defi}
\label{dfn: singular solution} If $q \in C^2((t_0, +\infty), \Xh)$ ( or $C^2((-\infty, t_0), \Xh)$) is a solution of \eqref{eq:NewtonINTRO}, we say
\begin{enumerate}
%\item $q(T^{\pm})$ is a \textbf{collision singularity}, if $T^{\pm} \in \rr$ and $\lim_{t \to T^{\pm}} q(t)= 0$;
\item $q(t)$ is a \textbf{parabolic motion}, if $\lim_{t \to \pm \infty} |q_i(t)-q_j(t)| =+\infty$, for all $ i \ne j$, and $\lim_{t \to \pm \infty} \qd_i(t) =0$, for all $i$;
\item $q(t)$ is a \textbf{hyperbolic motion}, if  $\lim_{t \to \pm \infty} |q_i(t)-q_j(t)| =+\infty$, for all $ i \ne j$,  and $\lim_{t \to \pm \infty}\qd_i(t)$ exists, for all $i$, and all the limits are different from each other.
\end{enumerate}
Moreover $q(\pm \infty)$ will be called a parabolic or hyperbolic infinity correspondingly.
\end{defi}

Under polar coordinates, when $q(\pm \infty)$ is a hyperbolic infinity, by Chazy \cite{Chazy29} (see \cite{DMMY19} for a modern treatment), $s(t)$ converges to some $s^{\pm} \in \E \setminus \Delta$, as $t \to \pm \infty$. On the other hand, when $q(\pm \infty)$ is a parabolic infinity, just like in the case of a total collision, $s(t)$ converges to the set of \emph{normalized central configurations}, as $t \to \pm \infty$ and it is not clear whether there will be a definite limit.

A natural question to ask is for which pairs of normalized configurations could there be a solution connecting them (after normalization) as time goes form negative infinity to positive infinity. This is related to the \emph{scattering theory} (see \cite{DMMY19}). Using action minimization method one can construct half of such a solution with either time goes to negative infinity or positive infinity (see \cite{MV09} and \cite{MV19}). It is impossible to find an entire solution defined on $\rr$ as a minimizer under Newtonian potential (see \cite{LM14}). Meanwhile for the spatial $n$-center problem, one can construct such entire solutions using some minimax approach \cite{BDT17}. To generalizes such a result to the $n$-body problem, it is also necessary to develop some Morse index theory for parabolic or hyperbolic motion.

With the above questions in mind, we introduce the following definition.

\begin{defi}
\label{defi:DoublyAsym} A solution $q \in C^2((T^-, T^+), \Xh)$ ($-\infty \le T^- < T^+ \le +\infty$) will be called a \textbf{doubly asymptotic solution}, if it satisfying the following conditions:

\begin{equation*}
\begin{cases}
q(T^{\pm}) \text{ is a total collision}, & \text{ if and only  } T^{\pm} \text{ is finite }; \\
q(T^{\pm}) \text{ is a parabolic/hyperbolic infinity}, & \text{ if and only if } T^{\pm} = \pm \infty.
\end{cases}
\end{equation*}
\end{defi}

\begin{rem}
In the above definition we did not consider solutions with partial collisions, as well as other types of final motion like elliptic-parabolic, elliptic-hyperbolic, parabolic-hyperbolic and so on that were listed by Chazy, as are not discussed in this paper.
\end{rem}

%One possible approach is using variational method to construct such singular solutions. However so far all these results are obtained on some simplified problem like anisotropic Kepler problem \cite{BTV13}, \cite{BTV14} or the $n$-center problem \red{[Add some references**********].}

Since the domain of a doubly asymptotic solution is non-compact, some care has to be taken when we try to define its Morse index.
For any $[t_1, t_2] \subset (T^-,T^+)$, $q|_{[t_1, t_2]}$ is a collision-free critical point of the action functional $\A$ in $H^1([t_1, t_2], \Xh)$. The Morse index of $q|_{[t_1, t_2]}$ in $H_0^1([t_1, t_2], \Xh)$, denoted by $m^-(q; t_1, t_2)$, is the dimension of the largest subspace in $H_0^1([t_1, t_2], \Xh)$, where the second derivative $d^2\mathcal{A}(q; t_1, t_2)$ is negative.

\begin{defi}
\label{dfn: Morse index} Let $\{t^{\pm}_k\}_{k \in \zz^+}$ be two sequences of times satisfying $ T^- < t^-_k< t^+_k < T^+$ and $\lim_{n \to \infty} t^{\pm}_k = T^{\pm}$. We define the \textbf{Morse index} of a doubly asymptotic solution $q \in C^2((T^-, T^+), \Xh)$ as
\begin{equation}
 \label{dfn: morse index}  m^-(q; T^-, T^+)=\lim_{k \to \infty}m^-(q; t^-_k, t^+_k).
\end{equation}
\end{defi}
\begin{rem}
Because of the following monotone property (see \cite{CH53} or \cite{HWY})
\begin{equation*} m^-(q; t_1,t_2)\leq m^-(q; t^*_1,t^*_2), \; \text{ if } \; t^*_1\leq t_1, t_2\leq t^*_2,
\end{equation*}
$m^-(q; T^-, T^+)$ is well-defined and independent of the choice of the sequences $\{t^{\pm}_k\}$. The above definition of Morse index is not a surprise. The main challenge is to compute it and this will be the main contribution of our paper.
\end{rem}

To describe our first result, for any $s \in \E \setminus \Delta$, set the Hessian of $U$ at $s$ restricted on $\E$ with respect to the Euclidean inner product as
\begin{equation} \label{eq: Hessian U}
 D^2U|_{\mathcal{E}}(s)=D^2U(s)+U(s)M,
\end{equation}
and with respect to the $M$ inner product as
\begin{equation}
\label{eq: normalized Hessian} M^{-1}D^2U|_{\mathcal{E}}(s)=M^{-1}D^2U(s)+U(s)I. \end{equation}

\begin{defi} \label{def: BS condition}

Given a central configuration $s_0 \in \E$. We denote the eigenvalues of $M^{-1}D^2U|_{\mathcal{E}}(s_0)$ by
$$ \lmd_1(s_0) \le \lmd_2(s_0) \le \cdots \le \lmd_{n^*-1}(s_0). $$
We say $s_0$ satisfies
\begin{enumerate}
\item the \textbf{spiral} condition, if $\lmd_1(s_0)<-\frac{1}{8}U(s_0)$,
\item the \textbf{non-spiral} condition, if $\lmd_1(s_0) \ge -\frac{1}{8}U(s_0)$;
\item the \textbf{strict non-spiral} condition, if $\lmd_1(s_0) > -\frac{1}{8} U(s_0)$.
\end{enumerate}
\end{defi}

\begin{rem} \label{rem: exp BS condition}
Up to our knowledge, the spiral condition first appeared in the study of the isosceles three body problem by Devaney \cite{Dv80} and Moeckel \cite{Mk81}, where it determines whether the corresponding stable or unstable manifold spirals into the equilibrium (corresponding to the Euler configuration) on the collision manifold.

On the other hand, the strict non-spiral condition, which is also called the [BS]-condition in \cite{BHPT}, first appeared in \cite{BS}, where Barutello and Secchi first proved that under this condition the Morse index of a collision solution, which approaches to the central configuration $s_0$ as it goes to the total collision, must be infinity.

%In \cite{BHPT}, the strict non-spiral condition was called [BS]-condition. We prefer the current name because of the following reason. In McGehee coordinates, see subsection \ref{subsec: McGehee coordinates}, central configurations give rise to equilibria on the collision manifold. The above conditions determine whether some of the eigenvalues of the linearized systems at the equilibria have non-zero imaginary parts, see Remark \ref{rem: spiral}, which then determine whether the corresponding stable and unstable manifolds spiral into and out of the equilibria. For details see \cite{Dv80} and \cite{Mk89}.
\end{rem}

\begin{thm}
\label{thm 1 singular} Given a doubly asymptotic solution $q \in C^2((T^-, T^+), \Xh)$ with both limits $\lim_{t \to T^{\pm}}s(t) = s^{\pm}$ exist.
\begin{enumerate}
\item[(a).] If $q(T^{\pm})$ is a total collision or a parabolic infinity with the corresponding $s^{\pm}$ satisfying the spiral condition, then for any $t_1 \in (T^-, T^+)$,
\begin{equation} \label{eq: MorseIndexLowBound}
\begin{aligned}
\lim_{t_{2} \to T^+}\frac{m^{-}(q; t_{1}, t_{2})}{|\ln\beta(t_{2})|}
& =\frac{1}{3\sqrt{2}\pi}\sum_{i=1}^{l^+}\sqrt{-\frac{1}{8}-\frac{\lambda_{i(s^{+})}}{U(s^{+})}}, \\
\lim_{t_{2} \to T^-}\frac{m^{-}(q; t_{2}, t_{1})}{|\ln\beta(t_{2})|}
& =\frac{1}{3\sqrt{2}\pi}\sum_{i=1}^{l^-}\sqrt{-\frac{1}{8}-\frac{\lambda_{i(s^{-})}}{U(s^{-})}},
\end{aligned}
\end{equation}

where  $l^{\pm}=\#\{1 \le i \le n^*-1:\lmd_{i}(s^{\pm})<-\frac{U(s^{\pm})}{8}\}$ and

\begin{equation} \label{eq: beta}
\beta(t)= \begin{cases}
|t-T^{\pm}|, \,   & \text{when $q(T^{\pm})$ is a total collision}, \\
 |t|, \,  & \text{when $q(T^{\pm})$ is a parabolic infinity}.
\end{cases}
\end{equation}
In particular this implies $m^-(q; T^-, T^+)= +\infty$.

\item[(b).] If $q(T^+)$ is either a hyperbolic infinity or a total collision/parabolic infinity with $s^+$ satisfying the strict non-spiral condition, and $q(T^-)$ is either a hyperbolic infinity or a total collision/parabolic infinity with $s^-$ satisfying the strict non-spiral condition, then $m^-(q; T^-, T^+)$ is finite.
\end{enumerate}
\end{thm}

\begin{proof}
Property (a) and (b) follow from Theorem \ref{thm: non-hyper} and \ref{thm: hyper} in section \ref{sec: index th} respectively.
\end{proof}

\begin{rem}
The result $m^-(q: T^-, T^+)=+\infty$ in property (a) of the above theorem was already obtained for a solution with a total collision in \cite{BS} or with a parabolic infinity in \cite{BHPT}. Compare with their approaches, the one used here allows us to achieve two additional things:
\begin{itemize}
\item we can given an estimate of how fast the Morse index grows when the solution approaches a total collision or a parabolic infinity as in \eqref{eq: MorseIndexLowBound};
\item we are able to give a simple and precise formula for the computation of the Morse index of homothetic solutions given in the next theorem.
\end{itemize}
\end{rem}

Let $q \in C^2((T^-, T^+), \Xh)$ be a homothetic solution, then $s(t) \equiv s_0$, $ \forall t$, for some normalized central configuration $s_0$. When $s_0$ satisfies the spiral condition, then property (a) in Theorem \ref{thm 1 singular} applies. Our next result shows what happens otherwise.

\begin{thm}
\label{thm 2 homo} Let $H_0$ be the energy of the homothetic solution $q(t)$. If the corresponding normalized central configuration $s_0$ satisfies the non-spiral condition,
\begin{equation*}
m^-(q; T^-, T^+)= \begin{cases}
m^-(M^{-1}D^2U|_{\mathcal{E}}(s_0)), & \text{ when } H_0 < 0;\\
0, & \text{ when } H_0 \ge 0,
\end{cases}
\end{equation*}
where $m^-(M^{-1}D^2U|_{\mathcal{E}}(s_0))$ is the number of negative eigenvalues of the matrix $M^{-1}D^2U|_{\mathcal{E}}(s_0)$.
\end{thm}

\begin{rem}
We point out that in Theorem \ref{thm 2 homo}, the central configuration $s_0$ only needs to satisfy the non-spiral condition, while in property (b) of Theorem \ref{thm 1 singular} the corresponding central configuration needs to satisfy the stronger strict non-spiral condition. The improvement is possible because we have precise expressions  for the homothetic solutions.
\end{rem}

%Although the definition of the Morse index in \cite{BHPT} is different from the one given in Definition \ref{dfn: morse index} and the precise relation between these two types of Morse indices is not clear.

%The above theorem is inspired by results from \cite{BS} and \cite{BHPT}. In \cite{BS} the authors studies solutions end at a collision singularity, and obtained a result similar to property (a). In \cite{BHPT}, solutions that either end in a collision singularity or approach to a parabolic or hyperbolic singularity were studied, and results similar as above were obtained. Notice that here we are considering solutions both start and end at a singularity.

%Our approach is closely related with the one in \cite{BHPT}. It is even possible to give a proof of Theorem \ref{thm 1 singular} based on the approach in \cite{BHPT}. Nevertheless we adopt a different approach here and one of the advantages of our approach is it allows us to give a precise formula for the Morse index of a given homothetic solution.

\begin{proof}
The theorem follows from Proposition \ref{prop: Mors index nonspirial}. \end{proof}

Our paper is organized as follows: in Section \ref{sec: McGehee coordinates}, the McGehee coordinates will be used to study the asymptotic behavior of the system near a total collision or a parabolic infinity and a variation of the McGehee coordinates (which will be called \emph{hyperbolic McGehee coordinates}) is introduced to understand the asymptotic behavior of the system near a hyperbolic infinity; in Section \ref{sec: index th}, we establish an index theory for doubly asymptotic solutions and prove Theorem \ref{thm 1 singular}; in Section \ref{sec: homo index}, we show how to compute the Morse index of a homothetic solution with arbitrary energy and prove Theorem \ref{thm 2 homo}.

\textbf{Notations}. Following notations will be adopted throughout the paper.
\begin{itemize}
\item $I$ is the identity matrix and $J= \left( \begin{array}{cc}
0 & -I \\ I & 0
\end{array} \right)$. The dimensions of these matrices will not always be the same, but can be easily found out through the context.
\item Given a function $f: \rr^k \to \rr$, $\nabla f$ represents the gradient of $f$ with respect to the Euclidean inner product expressed as a column vector and $D^2 f$ the Hessian of $f$. 
\item For any positive integer $k$, $\vd$ represents the \emph{momentum space} $\mathbb{R}^{k} \oplus \{0\}$, which corresponds to the Dirichlet boundary condition, and $\vn$ represents the \emph{configuration space} $\{0 \} \oplus \mathbb{R}^k$, which corresponds to the Neumann boundary condition.
\item Given a vector $\xi \in \rr^k$ , $\langle \xi \rangle$ will represent the linear subspace spanned by the vector.
\item Given a finite set $A$, $\# A$ will represent the number of elements in the set.
\end{itemize}

\section{McGehee coordinates and dynamics of the linear system along a doubly asymptotic solution}\label{sec: McGehee coordinates}

Throughout this section, let $q \in C^2((T^-, T^+), \Xh)$ be a doubly asymptotic solution with energy $H_0$ and satisfying
\begin{equation} \label{eq: lim s}
\lim_{t \to T^{\pm}} s(t)= \lim_{t \to T^{\pm}} q(t)/r(t) : = s^{\pm}.
\end{equation}

In order to compute the Morse index of such a solution, we need to understand its asymptotic behaviors. For a total collision or a parabolic infinity, this can be achieved by \emph{McGehee coordinates} (see \cite{Mg74} and \cite{Mk89}). For a hyperbolic infinity, we introduce a new set of coordinates following the spirit of McGehee, which will be called \emph{hyperbolic McGehee coordinates}.

\subsection{Asymptotic estimates.}
Since $\E$ is an $n^*:=d(n-1)$-dimension ellipse, we introduce a smooth coordinate chart $(\Omega, \psi)$ on $\E$,
\begin{eqnarray*}
\psi: \Omega \to \psi(\Omega) \subset \rr^{n^*-1}; \;\; s \mapsto x.
\end{eqnarray*}
Set $\U(x):= U(\psi^{-1}(x))$ and
\begin{equation}
\label{eq: Mhat} \Mh:= \left(\frac{\partial \psi^{-1}}{\partial x} \right)^T M \left(\frac{\partial \psi^{-1}}{\partial x} \right).
\end{equation}
\begin{rem} \label{rem: Mhat}
Notice that $\Mh$ depends on $x$ and $\Mh^T= \Mh$.
\end{rem}

Under the new variables $(r, x)$, the Lagrangian can be written as
\begin{align*}
L(r,x,\dot{r},\dot{x}) & =K(r, x, \dot{r}, \dot{x})+r^{-1}\U(x) \\
					   & =\frac{1}{2}(\dot{r}^2+r^2 \lan \Mh \xd, \xd \ran) + r^{-1} \U(x).
\end{align*}
Moreover set $x^{\pm} :=\psi(s^{\pm})$. Then
\begin{equation}
\label{eq: lim x} \lim_{t \to T^{\pm}} x(t) = \lim_{t \to T^{\pm}} \psi(s(t)) =x^{\pm},
\end{equation}
and
\begin{equation}
\label{eq: lim M hat} \hat{M}_{\pm} := \lim_{t \to T^{\pm}} \hat{M}(x(t)) = \hat{M}(x^{\pm}).
\end{equation}

Further introduce the new variables $(p_1, p_2)$ as
\begin{equation}
\label{eq: p1 p2} p_1=\dot{r}, \;\; p_2=r^2\hat{M}\dot{x}.
\end{equation}
We obtain the corresponding Hamiltonian
\begin{equation}
\label{eq: Hamiltonian function} H(p_1, p_2, r, x)= \frac{1}{2} \left( p_1^2+\frac{\langle\hat{M}^{-1}p_2,p_2\rangle}{r^2} \right)-\frac{\U(x)}{r}.
\end{equation}

Let $\zt(t)=(p_1,p_2,r,x)(t)$ represent the solution $q(t)$ in the new coordinates introduced as above. Then it satisfies the following Hamiltonian equation
\begin{equation} \label{eq: Hamiltonian equation}
\dot{\zt}=J \nabla H(\zt), \end{equation}
where
$$ \nabla H(\zt)=\left(p_1, \frac{p_2^T \hat{M}^{-1}}{r^2},  \frac{\U(x)}{r^2}-\frac{\langle\hat{M}^{-1}p_2,p_2\rangle}{r^3}, \frac{D_x\langle \hat{M}^{-1} p_2,p_2\rangle}{2r^2}-\frac{D_x \U(x)}{r}\right)^T. $$
Moreover the linearized system of \eqref{eq: Hamiltonian equation} along $\zt(t)$ is
\begin{equation} \label{eq: linearied}
\dot{\xi}(t)= JB(t) \xi(t) := JD^2H(\zt(t))\xi(t),
 \end{equation}
where
\begin{equation} \label{eq: H2} \begin{split}
& D^2H(\zt(t))= \\
&\left(
  \begin{array}{cccc}
    1 &   0 & 0 &   0\\
    0   & \frac{\hat{M}^{-1}}{r^2} &   -\frac{2}{r^3}\hat{M}^{-1}p_2  & \frac{D_x(\hat{M}^{-1}p_2)}{r^2}\\
    0 &   -\frac{2}{r^3}p_2^T\hat{M}^{-1} &  \frac{3\langle \hat{M}^{-1}p_2,p_2\rangle}{r^4}-\frac{2\U(x)}{r^3} &  \frac{D_x\U(x)}{r^2}-\frac{ D_x\langle \hat{M}^{-1}p_2,p_2\rangle}{r^3} \\
    0   & \frac{\nabla_{x}(p_2^T\hat{M}^{-1})}{r^2} &    \frac{\nabla_x\U(x)}{r^2}-\frac{ \nabla_x\langle \hat{M}^{-1}p_2,p_2\rangle}{r^3} & \frac{\langle \hat{M}^{-1}p_2,p_2\rangle_{xx}}{2r^2}-\frac{\U_{xx}(x)}{r}  \\
  \end{array}
\right) (t).
\end{split}
\end{equation}

\begin{lem} \label{lem: asy1}
\begin{enumerate}
\item[(a).] If $q(T^{\pm})$ is a total collision or a parabolic infinity,
\begin{eqnarray*}
\lim_{t\to T^{\pm}} r(t)\beta^{-\frac{2}{3}}(t) & = & [18U(s^{\pm})]^{\frac{1}{3}},\\
\lim_{t\to T^{\pm}} |\dot{r}(t)|\beta(t)^{\frac{1}{3}} & = & [2U(s^{\pm})/3]^{\frac{1}{3}},\\
\lim_{t\to T^{\pm}} r^{\frac{3}{2}}(t)|\dot{s}(t)|_M & = & 0,
\end{eqnarray*}
where $\beta(t)$ is defined as in \eqref{eq: beta}.

\item[(b).] If $q(T^{\pm})$ is a hyperbolic infinity,
\begin{eqnarray*}
\lim_{t\to T^{\pm}} r(t)|t|^{-1} & = & \sqrt{2H_0}, \\
\quad \lim_{t \to T^{\pm}} |\dot{r}(t)| & = & \sqrt{2H_0}, \\
\lim_{t\to T^{\pm}} r(t)|\dot{s}(t)|_M & = & 0.
\end{eqnarray*}
\end{enumerate}
\end{lem}

\begin{proof} (a). This is the well-known Sundman-Sperling estimate, detailed proofs can be found in \cite[Theorem 4.18]{BFT08} or \cite[Theorem 7.7]{BTV14}.

(b). Since $q(T^{\pm})$ is a hyperbolic infinity,  $T^{\pm}=\pm \infty$. We show the details for $T^+=+\infty$ ($T^-=-\infty$ can be proven similarly).

As $\I(q)= \lan Mq, q \ran$,
\begin{equation}
 \ddot{\I}(q)=2\lan M \qd, \qd \ran + 2\lan Mq, \ddot{q} \ran.
\end{equation}
By the homogeneity of $U$, $\lan Mq, \ddot{q} \ran = -U(q)$. Then
\begin{equation}
\ddot{\I}(q) = 4K(\qd) - 2 U(q) = 4H_0+2U(q).
\end{equation}
Recall that $\lim_{t \to +\infty}|q_i-q_j|=+\infty$, for all $i \ne j$,  $U(q(t)) \to 0$, as $t \to +\infty$. Therefore $ \lim_{t \to +\infty} \ddot{\I}(q(t)) = 4H_0.$ As a result,
\begin{equation}
\label{eq: I Id}  \lim_{t \to +\infty} \dot{\I}(q(t)) t^{-1} = 4H_0, \quad \lim_{t \to +\infty} \I(q(t)) t^{-2} = 2H_0.
\end{equation}
They immediately imply the first two identities in property (b), as $r(t)= \sqrt{\I(q(t))}$.

Since $s=(s_i)_{i=1}^n$ and $s_i= \I^{-\ey}q_i$,
$$ \qd_i= \I^{\ey}(q)\dot{s}_i +\ey \I^{-1}(q) \dot{\I}(q)q_i.$$
Then
$$ 2 K(\qd) = \lan M \qd, \qd \ran = \I(q) |\dot{s}|^2_M + \frac{1}{4} \I^{-1}(q) \dot{\I}^2(q). $$
As a result,
$$ \I(q) |\dot{s}|^2_M= 2K(\qd)-\frac{1}{4}\I^{-1}(q) \dot{\I}^2(q)= 2H_0+2U(q)-\frac{1}{4}\I^{-1}(q) \dot{\I}^2(q). $$
Since \eqref{eq: I Id} implies $\lim_{t \to +\infty} \frac{1}{4}\I^{-1}(q(t)) \dot{\I}^2(q(t))=2H_0$,
$$ \lim_{t \to +\infty} \I(q(t)) |\dot{s}(t)|^2_M = \lim_{t \to +\infty} U(q(t))= 0.$$
This then implies the third identity in property (b).

\end{proof}

%For the singular orbits with type I limit,  we use McGehee coordinate. In \cite{Dv78}, \cite{Dv82}, using McGehee coordinates \cite{Mg74}, Devaney blew up the singularity at the origin to a two dimensional collision manifold. He studied the dynamics on this manifold and gave a complete picture of it. We will recall some of these results that will be needed in our paper.

\subsection{McGehee coordinates.} \label{subsec: McGehee coordinates}
Following McGehee \cite{Mg74} we define the new coordinates $v$ and $u$ as
\begin{equation}
\label{eq: v u McGehee 1}  v=r^{1/2} p_1=r^{\frac{1}{2}}\dot{r},\;\; u=r^{-\frac{1}{2}}p_2 =r^{\frac{3}{2}}\hat{M}\dot{x}.
\end{equation}
Then equation \eqref{eq: Hamiltonian equation} becomes
\begin{equation}
 \label{eq: McGehee 1}
 \begin{cases}
 \dot{v} &= r^{-\frac{3}{2}}\big(\frac{1}{2}v^2+\langle\hat{M}^{-1}u,u\rangle-\U(x)\big), \\
 \dot{u} &= r^{-\frac{3}{2}} \big(-\frac{1}{2}uv+\U_x(x)-\frac{1}{2}\langle(\hat{M}^{-1})_xu,u\rangle \big), \\
 \dot{r} &= r^{-\frac{1}{2}}v, \\
 \dot{x} &= r^{-\frac{3}{2}} \hat{M}^{-1}u.
 \end{cases}
 \end{equation}
By changing the time parameter from $t$ to $\tau$ with $dt = r^{\frac{3}{2}} \, d\tau$, we have
\begin{equation}
\begin{cases}
v'&=\frac{1}{2}v^2+\langle\hat{M}^{-1}u,u\rangle-\U(x), \\
u'&=-\frac{1}{2}uv+\U_x(x)-\frac{1}{2}\langle(\hat{M}^{-1})_xu,u\rangle,\\
r'&=rv,\\
x'&=\hat{M}^{-1}u,
\end{cases} \label{eq: McGehee 1.1}
\end{equation}
where $'$ means $\frac{d}{d\tau}$ throughout the paper.

In these new coordinates, the energy identity becomes
\begin{equation}  \frac{\langle\hat{M}^{-1}u,u\rangle+v^2}{2}-\U(x)=rH_0. \label{eq: energy M1}      \end{equation}

\begin{rem}
\label{rem: spiral} Notice that $(\sqrt{2\U(x_0)}, 0, 0, x_0)$ is an equilibrium of \eqref{eq: McGehee 1.1}, where $x_0= \psi(s_0)$ with $s_0$ being a normalized central configuration. By a straight forward computation, if $s_0$ satisfies the spiral condition, at least one of the eigenvalues of the linearized system at the equilibrium will have non-zero imaginary part.
\end{rem}

\begin{lem}\label{lem:masy1}
If $q(T^{\pm})$ is a total collision or parabolic infinity,
\begin{enumerate}
\item[(a).] $\tau=\tau(t) \to \pm \infty$, as $t \to T^{\pm}$;
\item[(b).] $(|v|,u)(\tau) \to (\sqrt{2U(s^{\pm})}, 0)$, as $\tau \to \pm \infty$.
\end{enumerate}
\end{lem}

\begin{proof}
(a). By the definition of $\tau$ and Lemma \ref{lem: asy1},
$$ \frac{d \tau}{dt} = r^{-\frac{3}{2}}(t) \to (3 \sqrt{2U(s^{\pm})} \bt(t))^{-1}, \;\text{ as } t \to T^{\pm}. $$

This immediately implies property (a).

(b). Recall that $v=r^{\ey}\dot{r}$ and $u= r^{\frac{3}{2}}\hat{M} \dot{x}$. Lemma \ref{lem: asy1} implies
$$ \lim_{\tau \to \pm \infty} |v(\tau)|= \lim_{t \to T^{\pm}} r^{\ey}(t) |\dot{r}(t)| = \sqrt{2U(s^{\pm})}, $$
$$ \lim_{\tau \to \pm \infty} |u(\tau)|^2= \lim_{t \to T^{\pm}} r^3(t) \lan \hat{M}\dot{x}, \hat{M} \dot{x} \ran = \lim_{t \to T^{\pm}} r^3(t) |\dot{s}|^2_M =0. $$
\end{proof}

\begin{rem}
Since $v= r^{\ey} \dot{r}$ with $r>0$, the sign of $v$ is the same as $\dot{r}$. When $\tau \to \pm \infty$, we can determine the sign of $v(\tau)$ by simply checking whether the system is expending or shrinking. The same principle applies to Lemma \ref{lem:masy2} as well.
\end{rem}

From an index point of view, it is difficult to work with \eqref{eq: McGehee 1.1}, as it is not Hamiltonian. Hence we shall still work with the linear Hamiltonian system \eqref{eq: linearied}. However to use the information obtained through the McGehee coordinates, we need to make a couple of transformations of \eqref{eq: linearied}.

First we change the time parameter from $t$ to $\tau$. Then \eqref{eq: linearied} becomes
\begin{equation} \xi'(\tau)=J B(\tau) \xi(\tau):=r^{\frac{3}{2}}(\tau)JD^2 H(\zt(\tau))\xi(\tau),\label{eq: linearied tau} \end{equation}
where
\begin{equation} \label{eq: B tau} \begin{split}
& B(\tau)= \\
& \left(
  \begin{array}{cccc}
    r^{\frac{3}{2}} &   0 & 0 &   0\\
    0   & \frac{\hat{M}^{-1}}{\sqrt{r}} &   -\frac{2\hat{M}^{-1}p_2}{ (\sqrt{r})^3} & \frac{D_x(\hat{M}^{-1}p_2)}{\sqrt{r}}\\
    0 &   -\frac{2p_2^T\hat{M}^{-1}}{(\sqrt{r})^3} &  \frac{3\langle \hat{M}^{-1}p_2,p_2\rangle}{(\sqrt{r})^5}-\frac{2\U(x)}{(\sqrt{r})^3} &  \frac{\U_x(x)^T}{\sqrt{r}}- \frac{\langle (\hat{M}^{-1})_xp_2,p_2\rangle^T}{(\sqrt{r})^3} \\
    0   & \frac{\nabla_x(p^T_2\hat{M}^{-1})}{\sqrt{r}} &    \frac{\nabla_x\U(x)}{\sqrt{r}}- \frac{\nabla_x\langle \hat{M}^{-1}p_2,p_2\rangle}{(\sqrt{r})^3} & \frac{\langle \hat{M}^{-1}p_2,p_2\rangle_{xx}}{2\sqrt{r}}-r^{\frac{1}{2}}\U_{xx}(x)  \\
  \end{array}
\right).
\end{split}
\end{equation}

Next the variable $r$ will be separated from the system. For this we introduce the following lemma, whose proof is a straightforward computation.

\begin{lem}
\label{lem: phi R B}
Let $R(\tau)$ be a path of symplectic matrices depending on $\tau$.
If $\xi'(\tau)=J B(\tau) \xi(\tau)$, then $\eta(\tau)=R(\tau)\xi(\tau)$ satisfies $\eta'(\tau) =J\Phi_R(B)(\tau)\eta(\tau)$ with
\begin{equation}
\label{eq: phi R B} \Phi_R(B)(\tau):=-JR'(\tau)R^{-1}(\tau)+R^{-T}(\tau)B(\tau)R^{-1}(\tau).
\end{equation}
In particular, when $R(\tau)$ is a constant matrix, $\Phi_R(B)(\tau)=R^{-T}B(\tau)R^{-1}$.
\end{lem}

In the rest of the subsection, set $\hat{B}(\tau)=\Phi_R(B)(\tau)$ with $\Phi_R(B)$ given by \eqref{eq: phi R B} and
\begin{equation} \label{eq: R I}
R(\tau) =\text{diag} (r^{\frac{3}{4}},r^{-\frac{1}{4}}I,r^{-\frac{3}{4}},r^{\frac{1}{4}}I)(\tau)
\end{equation}
By the above lemma, $\eta(\tau)=R(\tau)\xi(\tau)$ satisfies
\begin{equation}
\eta'(\tau) =J\bh(\tau)\eta(\tau),
\end{equation}
where
\begin{equation} \label{eq: hat B tau} \begin{split}
&\bh(\tau)= \\
&\left(
  \begin{array}{cccc}
    1 &   0 & -\frac{3}{4}v &   0\\
    0   & \hat{M}^{-1} &   -2 \hat{M}^{-1}u & D_x(\hat{M}^{-1}u)+\frac{vI}{4}\\
    -\frac{3}{4}v &   -2 u^T\hat{M}^{-1}&  3\langle \hat{M}^{-1}u,u\rangle-2\U(x) &  D_x\U(x)-D_x\langle \hat{M}^{-1}u,u\rangle \\
    0   & \nabla_x(u^T\hat{M}^{-1})+\frac{vI}{4} &   \nabla_x\U(x)-\nabla_x \langle \hat{M}^{-1}u,u\rangle & \frac{1}{2}\langle \hat{M}^{-1}u,u\rangle_{xx}-\U_{xx}(x)  \\
  \end{array}
\right).
\end{split}
\end{equation}
Notice that by \eqref{eq: v u McGehee 1},
$$ v=r^{\frac{1}{2}}p_1=r'r^{-1}, \;\; u=r^{-\frac{1}{2}}p_2 =\hat{M}x'. $$

To simplify notation, in the rest of this subsection we set
\begin{equation}
\label{eq: tau x v star} \tau^*= \pm \infty, \;\; s^*=s^{\pm}, \;\; x^*=x^{\pm}, \;\; v^* = \pm\sqrt{2U(s^{\pm})}, \;\; \hat{M}_{*}= \hat{M}_{\pm},
\end{equation}
without specifying the $+$ or $-$.
This is because as $\tau$ goes to $-\infty$ or $+\infty$, we need to make the corresponding choices between $-$ and $+$ for $s^*, x^*$, $v^*$ and $\hat{M}_*$.

\begin{lem}
\label{lem: hat B McGehee I} If $q(T^{\pm})$ is a total collision or a parabolic infinity,
\begin{equation} \label{eq: hat B tau1} \hat{B}(\tau^*)= \lim_{\tau \to \tau^*} \hat{B} (\tau) =\left(
  \begin{array}{cccc}
    1 &   0 & -\frac{3}{4}v^* &   0\\
    0   & \hat{M}_*^{-1} &  0& \frac{1}{4}v^*I \\
    -\frac{3}{4}v^* &   0&  -2\U(x^*) &  0 \\
    0   & \frac{1}{4}v^*I &   0& -\U_{xx}(x^*)  \\
  \end{array}
\right).
\end{equation}
\end{lem}
\begin{proof}
The result follows from Lemma \ref{lem: asy1}, as it implies $\lim_{\tau \to \tau^*}(v,u)(\tau) =(v^*,0)$.
\end{proof}

Rewrite $\bh(\tau^*)$ as $\hat{B}(\tau^*)=\hat{B}_1(\tau^*)\diamond \hat{B}_2(\tau^*) $, where
\begin{equation}
 \label{eq: B hat 1 2 type I} \hat{B}_1(\tau^*)=\left(\begin{array}{cc}1 & -\frac{3}{4}v^*\\
                             -\frac{3}{4}v^*&-2\U(x^*)\end{array}\right), \;\;
                               \hat{B}_2(\tau^*)=\left(\begin{array}{cc}\hat{M}_*^{-1} & \frac{1}{4}v^*I\\
                             \frac{1}{4}v^*I&-\U_{xx}(x^*)\end{array}\right).
 \end{equation}
Here $\diamond$ represents the \emph{symplectic sum} introduced by Long (see \cite{Lon4}): for any two $2m_k\times 2m_k$ matrices, $O_k=\left(\begin{array}{cc}A_k&B_k\\
                             C_k&D_k\end{array}\right)$, $k=1, 2$,
\begin{equation} O_1 \diamond O_2=\left(
  \begin{array}{cccc}
   A_1 &   0 & B_1 &   0\\
                            0   & A_2 &   0 & B_2\\
                           C_1 &   0 & D_1 &   0\\
                           0   & C_2 &   0 & D_2  \\
  \end{array}
\right). \nonumber
\end{equation}

$J\bh(\tau^*)$ being hyperbolic or not will play crucial roles in our proofs. We will show this only depends on $s^*$ (or equivalently $x^*$). To prove this, we introduce two lemmas first.

\begin{lem}
\label{lem: eigenvalue U hat} Let $s_0$ be a normalized central configuration and $x_0= \psi(s_0)$. Then $\lmd_i(s_0)$, $i =1, \dots, n^*-1$ (see Definition \ref{def: BS condition}), are eigenvalues of $A^T \U_{xx}(x_0)A$, where $A$ is a matrix satisfying $A^T \Mh_0 A=I$ and
$$  \Mh_0 = \left. \left( \frac{\partial \psi^{-1}}{\partial x} \right) \right|_{x= x_0}^T M \left. \left( \frac{\partial \psi^{-1}}{\partial x} \right) \right|_{x= x_0}.$$
\end{lem}

\begin{proof}
By \eqref{eq: Mhat} and the fact $\U_{xx}(x_0)=\left(\frac{\partial \psi^{-1}}{\partial x} \right)^T D^2 U|_{\mathcal{E}}(s_0)\left(\frac{\partial \psi^{-1}}{\partial x} \right)$,
\begin{align*}
A^T \U_{xx}(x_0)A  & = A^{-1}\Mh_0^{-1} \U_{xx}(x_0) A \\
					& = A^{-1}\left(\frac{\partial \psi^{-1}}{\partial x} \right)^{-1}M^{-1}\left(\frac{\partial \psi^{-1}}{\partial x} \right)^{-T} \U_{xx}(x_0)A \\
					 & = A^{-1} \left(\frac{\partial \psi^{-1}}{\partial x} \right)^{-1} M^{-1}D^2U|_{\E}(s_0)\left(\frac{\partial \psi^{-1}}{\partial x} \right) A.
\end{align*}
 \end{proof}

\begin{lem}
\label{lem: symplectic similarity} Given a matrix $ P= \left( \begin{array}{cc} I & Q \\ Q^T & Q^T Q -R
\end{array} \right).$
If $Q$ is symmetric, $JP$ is symplectic similar to $J\hat{P}$ with $ \hat{P} = \left( \begin{array}{cc} I & 0 \\ 0 & -R
\end{array}\right).$
\end{lem}
\begin{proof}
Since $Q= Q^T$, $O= \left( \begin{array}
{cc} I & -Q \\ 0 & I
\end{array} \right)$ is symplectic. A direct computation shows $O^T P O= \hat{P}$.
As a result,
$$ J\hat{P}= JO^T PO= O^{-1}(JP) O. $$
\end{proof}

\begin{prop}
\label{prop: hyp type I}
If $q(T^{\pm})$ is a total collision or a parabolic infinity,  $J\hat{B}(\tau^*)$ is hyperbolic if and only if $s^*$ satisfies the strict non-spiral condition.
\end{prop}

\begin{proof} Since $J\bh(\tau^*)=J\bh_1(\tau^*) \diamond J\bh_2(\tau^*)$, it is hyperbolic if and only if $J\bh_1(\tau^*)$ and $J\bh_2(\tau^*)$ are hyperbolic. A direct computation shows the eigenvalues of $J \bh_1(\tau^*)$ are $\pm \frac{5}{2\sqrt{2}}\sqrt{\U(x^*)}$.

For $J\bh_2(\tau^*)$, we can always find an invertible matrix $A$ satisfying $A^T \Mh_* A= I$. Meanwhile $A_d=\text{diag}(A^T, A^{-1})$ is a symplectic matrix. Therefore
$$ A_d (J\bh_2(\tau^*))A_d^{-1}= J A_d^{-T} \bh_2(\tau^*)A_d^{-1} =J \Phi_{A_d}(\bh_2(\tau^*)), $$
where $\Phi_{A_d}(\bh_2(\tau^*)$ is defined as in \eqref{eq: phi R B}. This means $J\bh_2(\tau^*)$ is hyperbolic if and only if $J \Phi_{A_d}(\bh_2(\tau^*))$ is.

Now notice that
$$ \Phi_{A_d}(\bh_2(\tau^*)) = \left( \begin{array}
{cc} I & \frac{1}{4}v^* I \\ \frac{1}{4}v^* I & -A^T \U_{xx}(x^*)A
\end{array} \right).
$$
Let $P= \Phi_{A_d}(\bh_2(\tau^*))$ and $O=\left( \begin{array}
{cc} I & -\frac{1}{4}v^*I \\ 0 & I
\end{array} \right)$. By Lemma \ref{lem: symplectic similarity},
$$ O^{-1} J \Phi_{A_d}(\bh_2(\tau^*)) O=J \hat{P}, \text{ where } \hat{P}=\left( \begin{array}
{cc} I & 0 \\ 0 & - \big( \frac{(v^*)^2}{16}I +A^T \U_{xx}(x^*)A \big)
\end{array} \right).$$
Recall that $v^*= \pm \sqrt{2 \U(x^*)}= \pm \sqrt{2U(s^*)}$, by a simple computation,
$$ (J\hat{P})^2 = \left( \begin{array}
{cc} \frac{U(s^*)}{8}I +A^T \U_{xx}(x^*)A & 0 \\ 0 & \frac{U(s^*)}{8}I +A^T \U_{xx}(x^*)A
\end{array} \right).$$
By Lemma \ref{lem: eigenvalue U hat}, the eigenvalues of $J\hat{P}$ are
$$ \pm \sqrt{\frac{U(s^{*})}{8}+\lmd_i(s^*)}, \;\; i=1, \dots, n^*-1. $$
This implies the desired result.

\end{proof}

\subsection{Hyperbolic McGehee coordinates.} To deal with the hyperbolic infinity,  we introduce the following coordinates
\begin{equation}
\label{eq: v u McGehee 2}  v= p_1=\dot{r},\;\; u=r^{-1}p_2 =r\hat{M}\dot{x}.
\end{equation}
Then equation \eqref{eq: Hamiltonian equation} becomes
\begin{equation}
 \label{eq: hyp McGehee 1}
 \begin{cases}
 \dot{v} &= r^{-1}\big(\langle\hat{M}^{-1}u,u\rangle-r^{-1}\U(x)\big), \\
 \dot{u} &= r^{-1} \big(-uv+r^{-1}\U_x(x)-\frac{1}{2}\langle(\hat{M}^{-1})_xu,u\rangle \big), \\
 \dot{r} &= v, \\
 \dot{x} &= r^{-1} \hat{M}^{-1}u.
 \end{cases}
 \end{equation}
Like the previous subsection, we change the time parameter from $t$ to $\tau$, but this time with $dt = r \, d\tau$
\begin{equation}
\begin{cases}
v'&=\langle\hat{M}^{-1}u,u\rangle-r^{-1}\U(x), \\
u'&=-uv+r^{-1}\U_x(x)-\frac{1}{2}\langle(\hat{M}^{-1})_xu,u\rangle,\\
r'&=rv,\\
x'&=\hat{M}^{-1}u.
\end{cases} \label{eq: McGehee r1}
\end{equation}

The energy identity now becomes
\begin{equation}
\frac{1}{2}(\langle\hat{M}^{-1}u,u\rangle+v^2)-r^{-1}\U(x)=H_0. \label{eq: energy M2}
\end{equation}

\begin{lem}\label{lem:masy2}
If $q(T^{\pm})$ is a hyperbolic infinity, 
\begin{enumerate}
\item[(a).] $\tau=\tau(t) \to \pm \infty$, as $t \to T^{\pm}$;
\item[(b).] $(|v|,u)(\tau) \to (\sqrt{2H_0}, 0)$, as $\tau \to \pm \infty$.
\end{enumerate}
\end{lem}
\begin{proof}
(a). Since $q(T^{\pm})$ is a hyperbolic infinity,  the energy constant $H_0 >0$. By Lemma \ref{lem: asy1},
$$ \frac{d \tau}{dt} = r^{-1}(t) \to  \frac{|t|}{\sqrt{2H_0}}, \; \text{ as } t \to T^{\pm}. $$
Combining this with the fact that $T^{\pm}= \pm \infty$ (see Definition \ref{dfn: singular solution}), we get the desired property.

(b). By \eqref{eq: v u McGehee 2} and Lemma \ref{lem: asy1}, we have
$$ \lim_{\tau \to \pm \infty} |v(\tau)|= \lim_{t \to T^{\pm}} |\dot{r}(t)|= \sqrt{2H_0};$$
$$ \lim_{\tau \to \pm \infty} |u(\tau)|^2= \lim_{t \to T^{\pm}} r^2(t) \lan \hat{M} \dot{x}, \hat{M} \dot{x} \ran = \lim_{t \to T^{\pm}} r^2(t) |\dot{s}(t)|^2_M = 0. $$
\end{proof}

Like in the previous subsection, we change the time parameter in the linear system \eqref{eq: linearied} from $t$ to $\tau$. Now as $dt= r \,d\tau$, we get
\begin{equation}
\label{eq: linear McGehee II}   \xi'(\tau)=JB(\tau) \xi(\tau):=r(\tau)J\nabla^2 H(\zt(\tau))\xi(\tau).
\end{equation}
For the rest of this subsection, set $\hat{B}(\tau)=\Phi_R(B)$ with $\Phi_R(B)$ given by \eqref{eq: phi R B} and
\begin{equation} \label{eq: R II}
R(\tau) =\text{diag} (r^{\frac{1}{2}},r^{-\frac{1}{2}}I,r^{-\frac{1}{2}},r^{\frac{1}{2}}I)(\tau).
\end{equation}
By Lemma \ref{lem: phi R B},
\begin{equation} \label{eq: hat B tau 2} \begin{split}
&\hat{B}(\tau)= \\
&\left(
  \begin{array}{cccc}
    1 &   0 & -\frac{1}{2}v &   0\\
    0   & \hat{M}^{-1} &   -2 \hat{M}^{-1}u & D_x(\hat{M}^{-1}u)+\frac{1}{2}vI\\
    -\frac{1}{2}v &   -2 u^T\hat{M}^{-1}&  3\langle \hat{M}^{-1}u,u\rangle-\frac{2\U(x)}{r} &  \frac{D_x\U(x)^T}{r}- D_x \langle \hat{M}^{-1}u,u \rangle \\
    0   & \nabla_x(\hat{M}^{-1}u)+\frac{1}{2}vI &  \frac{\nabla_x \U(x)}{r}- \nabla_x \langle \hat{M}^{-1}u,u\rangle & \frac{1}{2}\langle \hat{M}^{-1}u,u\rangle_{xx}-\frac{\U_{xx}(x)}{r}  \\
  \end{array}
\right).
\end{split}
\end{equation}

Make the same assumption in notations as in \eqref{eq: tau x v star} for the rest of this subsection.
\begin{lem}
\label{lem: hat B McGehee II} If $q(T^{\pm})$ is a hyperbolic infinity,
\begin{equation} \label{eq: hat B tau11} \hat{B}(\tau^*)= \lim_{\tau \to \tau^*} \hat{B} (\tau) =\left(
  \begin{array}{cccc}
    1 &   0 & -\frac{1}{2}v^* &   0\\
    0   & \hat{M}_*^{-1} &   0& \frac{1}{2}v^*I\\
    -\frac{1}{2}v^* &   0&  0 & 0 \\
    0   &\frac{1}{2}v^*I &  0 &0  \\
  \end{array}
\right).
\end{equation}
\end{lem}
\begin{proof}
Since $r(\tau) \to +\infty$, when $\tau \to \pm \infty$. The desired result follows from Lemma \ref{lem: asy1}, as it implies $\lim_{\tau \to \tau^*} (v, u)(\tau)= (v^*, 0)$.
\end{proof}
As a result, we may write $\bh(\tau^*) = \bh_1(\tau^*) \diamond \bh_2(\tau^*)$ with
\begin{equation}
\label{eq: bh 1 2 type II} \bh_1(\tau^*)= \left( \begin{array}{cc}
1 &  -\ey v^* \\ -\ey v^* & 0
\end{array} \right), \;\; \bh_2(\tau^*) = \left( \begin{array}{cc}
\Mh_*^{-1} & \ey v^* I \\ \ey v^* I & 0
\end{array}\right).
\end{equation}

By a similar argument as in the proof of Proposition \ref{prop: hyp type I}, we find
\begin{prop}
\label{prop bh hyperbolic type II} If $q(T^{\pm})$ is a hyperbolic infinity, $J\bh(\tau^*)$ is always hyperbolic.
\end{prop}

%%%%%%%%%%%%%%%%%%%%%%%%%%%%%%%

\section{ A Morse index theory for doubly asymptotic solutions} \label{sec: index th}

A proof of Theorem \ref{thm 1 singular} will be given in this section. Since the Morse index of a solution is an integer associated to an unbounded Fredholm operator in an infinite dimension space, it is difficult to compute it directly. Meanwhile for a Hamiltonian system it is known the Morse index is related with the Maslov index, which is an intersection index between two paths of Lagrangian subspaces. In the following we give a brief introduction of the Maslov index that will be needed in our proofs.

Given a $2k$-dimension symplectic space $(\mathbb{R}^{2k},\omega)$, where $\om$ represents the standard symplectic form. Let $Lag(2k)$ represents the \emph{Lagrangian Grassmannian}, i.e. the space of all Lagrangian subspaces in $(\mathbb{R}^{2k},\omega)$. For any two continuous paths $L_1(t),L_2(t)$, $t\in[a,b]$, in $Lag(2k)$, the Maslov index $\mu(L_1,L_2; [a, b])$ is an integer invariant, uniquely determined by a set of axioms listed as Property I to VII below (for the details see \cite{CLM}).

\textbf{Property I. (Reparametrization invariance)}  Let $\vr:[c,d]\rightarrow [a,b]$ be a continuous and piecewise smooth function satisfying $\vr(c)=a$, $\vr(d)=b$, then \begin{equation}
\mu(L_1(t), L_2(t))=\mu(L_1(\vr(\tau)), L_2(\vr(\tau))). \label{adp1.1}
\end{equation}

\textbf{Property II. (Homotopy invariant with end points)} If two continuous
families of Lagrangian paths $L_1(s,t)$, $L_2(s,t)$, $0\leq s\leq 1$, $a\leq
t\leq b$ satisfy $\text{dim}(L_1(s,a)\cap L_2(s,a))=C_1, \text{dim}(L_1(s,b)\cap L_2(s,b)) = C_2,$, for any $0\leq s\leq 1$, where $C_1, C_2$ are two constant integers, then
\begin{equation}
\mu(L_1(0,t), L_2(0,t))=\mu(L_1(1,t),L_2(1,t)). \label{adp1.2}
\end{equation}

\textbf{Property III. (Path additivity)}  If $a<c<b$, then
\begin{equation}
\mu(L_1(t),L_2(t))=\mu(L_1(t),L_2(t); [a,c])+\mu(L_1(t),L_2(t); [c,b]).
\label{adp1.3}
\end{equation}

\textbf{Property IV. (Symplectic invariance)} Let $\ga(t)$, $t\in[a,b]$ be a
continuous path of symplectic matrices in $\Sp(2n)$, then
\begin{equation}
\mu(L_1(t),L_2(t))=\mu(\ga(t)L_1(t), \ga(t)L_2(t)). \label{adp1.4}
\end{equation}

\textbf{Property V. (Symplectic additivity)} Let $W_i$, $i=1,2$, be two symplectic spaces, if $L_i \in C([a,b], Lag(W_1))$ and $\hat{L}_i \in C([a,b], Lag(W_2))$, $i =1, 2$, then
\begin{equation}
\mu(L_1(t)\oplus \hat{L}_1(t),L_2(t)\oplus \hat{L}_2(t))= \mu(L_1(t),L_2(t))+
\mu(\hat{L}_1(t),\hat{L}_2(t)).
\label{adp1.5add}
\end{equation}

\textbf{Property VI. (Symmetry)} If $L_i \in C([a,b], \text{Lag}(2n))$, $i=1,2$, then
\begin{equation}  \mu(L_1(t), L_2(t))= \dim L_1(a)\cap L_2(a)-\dim L_1(b)\cap L_2(b)  -\mu(L_2(t),L_1(t)).    \label{adp1.5sym}
\end{equation}

In the case $L_1(t)\equiv V_0$,  $L(t)=\ga(t)V$, where $\ga$ is a path of  symplectic matrix we have
a monotonicity property (cfr. \cite{HOW15}).

\textbf{Property VII (Monotone property)}   Suppose for $j=1,2$,   $L_j(t)=\ga_j(t)V$,
where  $\dot{\ga}_j(t)=JB_j(t)\ga_j(t)$ with $\ga_j(t)=I_{2n}$.
If $B_1(t)\geq B_2(t)$ in the sense that $B_1(t)-B_2(t)$ is non-negative matrix,
then for any $V_0,V_1\in Lag(2n)$, we have
\bea  \mu(V_0, \ga_1V_1)\geq  \mu(V_0, \ga_2V_1).    \lb{adp1.5monotone}       \eea

An efficient way to study the Maslov
index is via crossing form introduced in \cite{RS1}.
 For simplicity and since it is  enough for our purpose, we only
review the case of the Maslov index for a path of Lagrangian subspace
with respect to a fixed Lagrangian subspaces. Let $\Lambda(t)$ be a
$C^1$-curve of Lagrangian subspaces with $\Lambda(0)=\Lambda$, and
let $V$ be a fixed Lagrangian subspace which is transversal to
$\Lambda$. For $v\in \Lambda$ and small $t$, define $w(t)\in V$ by
$v+w(t)\in \Lambda(t)$. Then the form \bea
Q(v)=\left.\frac{d}{dt}\right|_{t=0}\omega(v,w(t)) \lb{1.3a} \eea is
independent of the choice of $V$ (cfr.\cite{RS1}).  A crossing for
$\Lambda(t)$ is some $t$ for which $\Lambda(t)$ intersects $W$
nontrivially.  At each crossing, the crossing form is
defined to be \bea \Gamma(\Lambda(t),W,t)=Q|_{\Lambda(t)\cap W}.
\lb{1.3b} \eea A crossing is called {\it regular} if the crossing
form is non-degenerate. If the path is given by
$\Lambda(t)=\gamma(t)\Lambda$ with $\gamma(t)\in \Sp(2n)$ and
$\Lambda\in Lag(2n)$, then the crossing form is
equal to $(-\gamma(t)^TJ\dot{\gamma}(t)v,v)$, for $v\in
\gamma(t)^{-1}(\Lambda(t)\cap W)$, where $(\,,\,)$ is the standard
inner product on $\mathbb{R}^{2n}$.

For $\Lambda(t)$  and $W$ as above if the path has only regular
crossings, following \cite{LZ}, the Maslov index is equal to
$$ \mu(W,\Lambda(t))=m^+(\Gamma(\Lambda(a),W,a))+\sum_{a<t<b}  \text{sign}
(\Gamma(\Lambda(t),W,t))-m^-(\Gamma(\Lambda(b),W,b)),$$
where the sum runs  all over the  crossings $t\in(a,b)$ and $m^+,
m^-$ are the dimensions of  positive and negative definite
subspaces, $sign=m^+-m^-$ is the signature. We note that for a
$C^1$-path $\Lambda(t)$ with fixed end points, and it can be made to only
have regular crossings by a small perturbation.

When the Hamiltonian system is given by the Legendre transformation of a Sturm-Liouville system, since all the crossings are positive, we have
 \begin{equation*}
\mu(\vd,\Lambda(t))=\text{dim}(\Lambda(a)\cap \vd)+\sum_{a<t<b} \text{dim}(\Lambda(t)\cap \vd).
\end{equation*}
For more details see   \cite{RS1}, \cite{HO}.

%for $W=V_d$, which is the Lagrangian subspace corresponding to the Dirichlet boundary condition (for the detail see \cite{HO}).

Given a Lagrangian path $t \mapsto\Lambda(t)$, the difference of the Maslov indices of it with respect to two Lagrangian subspaces $V_0,V_1 \in \text{Lag}(2n)$, is
given in terms of the \emph{H\"{o}rmander index} (see \cite[Theorem 3.5]{RS1})
\begin{equation}
s(V_0, V_1; \Lambda(0), \Lambda(1)) = \mu(V_0, \Lambda(t))-\mu(V_1,\Lambda(t)).
\end{equation}
%Obviously for $\varepsilon>0$ small enough,
%\begin{equation} s(V_0, V_1; \Lambda(0), \Lambda(1)) = s(V_0, V_1; e^{-\varepsilon J}\Lambda(0),e^{-\varepsilon J}\Lambda(1)), \label{hp}
%\end{equation}

The H\"{o}rmander index is independent of the choice of the path connecting $\Lambda(0)$ and $\Lambda(1)$, and it satisfies
\begin{equation}
|s(V_0,V_1; \Lambda(0),\Lambda(1))|\leq 2n.  \label{hormd}
\end{equation}

Let $q \in C^2((T^-, T^+), \hat{\X})$ be a doubly asymptotic solution. After the above explanation, we are ready to study the index of $q(t)$.  First the linearized Hamiltonian system along such a solution was given in \eqref{eq: linearied}.
We denote its the fundamental solution as $\ga(t,t_1)$, i.e.
\[\dot{\ga}(t,t_1)=JB(t)\ga(t,t_1),\quad \ga(t_1,t_1)=I. \]

Recall that for any $[t^-, t^+] \subset (T^-, T^+)$, $m^-(q; t^-, t^+)$ represents the Morse index of $q|_{[t^-, t^+]}$. By the \emph{Morse index theorem} (see \cite{HS09}),
\begin{equation}
m^-(q; t^-, t^+)+n^*=\mu(\vd, \gamma(t,t^-) \vd; [t^-,t^+]),\label{Mor-Mas}
\end{equation}
where the right hand side is the Maslov index of the two paths $\vd$ and $\gamma(t,t^-)\vd$, $t \in [t^-, t^+]$.

In order to compute the limit of the above Morse index as $t^{\pm}$ converge to $T^{\pm}$, the results obtained in the previous section using (hyperbolic) McGehee coordinates will be crucial. Recall that the change of coordinates transfers  \eqref{eq: linearied} to the following system
\begin{equation} \label{eq: linear Ham tau}
\eta'(\tau) = J \bh(\tau) \eta(\tau), \; \text{ where } \; \bh(\tau) = \Phi_R(B)(\tau).
\end{equation}
Recall that $\Phi_R(B)(\tau)$ is defined in \eqref{eq: phi R B} with $R(\tau)$ given by \eqref{eq: R I} in McGehee coordinates and by \eqref{eq: R II} in hyperbolic McGehee coordinates.
The fundamental solution of \eqref{eq: linear Ham tau} will be denoted as $\hat{\gm}(\tau, \tau_1)$.

Next lemma shows the Maslov index is invariant under the change of the (hyperbolic) McGehee coordinates.
\begin{lem} \label{lem: Maslov index t tau} For any $T^- < t_1 < t_2 < T^+$, in (hyperbolic) McGehee coordinates,
$$ \mu(\vd, \hat{\gamma}(\tau,\tau_1)\vd; [\tau_1,\tau_2]) =\mu(\vd, \gm(t, t_1)\vd; [t_1, t_2]), $$
where $\tau_i=\tau(t_i)$, $i=1,2$.
\end{lem}

\begin{proof}
Recall that $\hat{\gm}(\tau,\tau_{1})=R(\tau)\gm(\tau,\tau_{1})R^{-1}(\tau_{1})$, where $R(\tau)$ is either \eqref{eq: R I} or \eqref{eq: R II}.
Then $R(\tau)\vd=\vd, R^{-1}(\tau)\vd=\vd$ and
\begin{align*}
\mu(\vd,\hat{\gm}(\tau,\tau_{1})\vd; [\tau_1, \tau_2]) & = \mu(\vd,R(\tau)\gm(\tau,\tau_{1})R^{-1}(\tau_{1})\vd; [\tau_1, \tau_2]) \\
 & =\mu(R^{-1}(\tau)\vd,\gm(\tau,\tau_{1}) R^{-1}(\tau_{1})\vd; [\tau_1, \tau_2]) \\
 & = \mu(\vd,\gm(\tau,\tau_{1})\vd; [\tau_1, \tau_2])
\end{align*}

Since the Maslov index is invariant under the change of time parameter, we have
$$ \mu(\vd, \gm(t, t_1)\vd; [t_1, t_2])=\mu(\vd, \hat{\gm}(\tau, \tau_1)\vd; [\tau_1, \tau_2]). $$

%Since $\hat{V}^{u}(\tau_{1})=\{v \in \mathbb{R}^{2k}|\lim_{\tau \to -\infty}\hat{\gamma}(\tau,\tau_{1}) v=0\}
%=\{v \in \mathbb{R}^{2k}|\lim_{\tau \to -\infty}R(\tau)\gm(\tau,\tau_{1})R^{-1}(\tau_{1}) v=0\}$
\end{proof}

\subsection{The spiral case.} We will prove Property (a) of Theorem \ref{thm 1 singular} in this subsection. For this we need the following lemma.

\begin{lem}\label{lem3.1}
For any $[t_1, t_2] \subset (T^-, T^+)$ and any $\hat{t} \in (t_1, t_2)$,
\bea
|m^{-}(q;t_{1},t_{2})-m^{-}(q;t_{1},\hat{t})-m^{-}(q;\hat{t},t_{2})|\leq 3n^{*}.\nonumber
\eea
\end{lem}
\begin{proof}
Let $\Lambda_{s}, s\in[0,1]$, be a continuous path of Lagrangian subspaces satisfying $\Lambda_{0}=\gamma(\hat{t},t_{1})\vd$ and $\Lambda_{1}=\vd$.
By the homotopy invariant property of the Maslov index,
\begin{align*}
\mu(\vd,\Lambda_{s};[0,1]) & +\mu(\vd,\gamma(t,\hat{t})\vd;[\hat{t},t_{2}])\\
 &=\mu(\vd,\gamma(t,\hat{t})\Lambda_{0};[\hat{t},t_{2}])+\mu(\vd,\gamma(t_{2},\hat{t})\Lambda_{s};[0,1]).
\end{align*}
Then
\begin{equation}  \label{bound}
\begin{split}
\mu(\vd,\gamma(t,\hat{t})\vd; & \; [\hat{t},t_{2}])-\mu(\vd,\gamma(t,\hat{t})\Lambda_{0};[\hat{t},t_{2}])\\
&=\mu(\vd,\gamma(t_{2},\hat{t})\Lambda_{s};[0,1])-\mu(\vd,\Lambda_{s};[0,1])\\
&=\mu(\gamma(\hat{t},t_{2})\vd,\Lambda_{s};[0,1])-\mu(\vd,\Lambda_{s};[0,1])\\
&=s(\gamma(\hat{t},t_{2})\vd,\vd;\Lambda_{0},\vd).
\end{split}
\end{equation}
By the Morse index theorem (\ref{Mor-Mas}),
\begin{equation} \label{eq: m=mu+mu}
\begin{aligned}
m^{-}(q;t_{1},t_{2})+n^{*} & =\mu(\vd,\gamma(t,t_{1})\vd; [t_{1},t_{2}]) \\
     &=\mu(\vd,\gamma(t,t_{1})\vd; [t_{1},\hat{t}])+\mu(\vd,\gamma(t,\hat{t})\Lambda_{0}; [\hat{t},t_{2}]).
\end{aligned}
\end{equation}
Meanwhile
\begin{equation*}
m^{-}(q;t_{1},\hat{t})+m^{-}(q;\hat{t},t_{2})+2n^{*} =\mu(\vd,\gamma(t,t_{1})\vd; [t_{1},\hat{t}])+\mu(\vd,\gamma(t,\hat{t})\vd; [\hat{t},t_{2}]).
\end{equation*}
Combining this with (\ref{bound}), \eqref{eq: m=mu+mu} and (\ref{hormd}), we get
\begin{equation*}
\begin{aligned}
|m^{-}(q;t_{1},t_{2})-m^{-}(q;t_{1},\hat{t})& -m^{-}(q;\hat{t},t_{2})| \\
& \leq n^{*}+|s(\gamma(\hat{t},t_{2})\vd,\vd;\Lambda_{0},\vd)|\leq 3n^{*}.
\end{aligned}
\end{equation*}
This completes the proof.
\end{proof}

The following theorem is a restatement of property (a) in Theorem \ref{thm 1 singular}.
\begin{thm}
\label{thm: non-hyper}
 If $q(T^{\pm})$ is a total collision or a parabolic infinity  with the corresponding $s^{\pm}$ satisfying the spiral condition, then for any $t_1 \in (T^-, T^+)$,
\begin{equation*}
\begin{aligned}
\lim_{t_{2} \to T^+}\frac{m^{-}(q; t_{1}, t_{2})}{|\ln\beta(t_{2})|}
& =\frac{1}{3\sqrt{2}\pi}\sum_{i=1}^{l^+}\sqrt{-\frac{1}{8}-\frac{\lambda_{i(s^{+})}}{U(s^{+})}}, \\
\lim_{t_{2} \to T^-}\frac{m^{-}(q; t_{2}, t_{1})}{|\ln\beta(t_{2})|}
& =\frac{1}{3\sqrt{2}\pi}\sum_{i=1}^{l^-}\sqrt{-\frac{1}{8}-\frac{\lambda_{i(s^{-})}}{U(s^{-})}},
\end{aligned}
\end{equation*}

where  $l^{\pm}=\#\{1 \le i \le n^*-1:\lmd_{i}(s^{\pm})<-\frac{U(s^{\pm})}{8}\}$ and $\beta(t)$ satisfies \eqref{eq: beta}.
\end{thm}

\begin{proof}
Let's assume $q(T^+)$ is a total collision or a parabolic infinity  with $s^{+}$ satisfying the spiral condition (the case for $q(T^-)$ can be proven similarly).

To keep the notations consistent with those in Section \ref{sec: McGehee coordinates}. Let
$$ \tau^*= +\infty, \; s^*= s^+, \; v^*= \sqrt{2U(s^+)}, \; x^*= \psi(s^+), \; \hat{M}_*= \hat{M}_+ =\hat{M}(x^*). $$
Then
$$ \lim_{\tau\rightarrow\tau^{*}}\hat{B}(\tau^{*})=\hat{B}_{1}(\tau^{*})\diamond\hat{B}_{2}(\tau^{*}), $$
where $\hat{B}_{1}(\tau^{*}), \hat{B}_{2}(\tau^{*})$ are given by (\ref{eq: B hat 1 2 type I}).

For $\varepsilon>0$ small enough (its precise value will be given later), there exists a $\tau_{\vep}$ (depending on $\vep$) large enough,  such that $\forall \tau>\tau_{\vep}$,
$$ \hat{B}_{1}(\tau^{*})\diamond\hat{B}_{2}(\tau^{*})-\varepsilon I_{2}
\diamond E_{\hat{M}_{*}} \le \hat{B}(\tau)\leq\hat{B}_{1}(\tau^{*})\diamond\hat{B}_{2}(\tau^{*})+\varepsilon I_{2}
\diamond E_{\hat{M}_{*}}. $$
where $E_{\hat{M}_*}=diag(\hat{M}_*^{-1},\hat{M}_*)$.

Let $\hat{\gm}(\tau, \tau_{\vep})$ be the fundamental solution of \eqref{eq: linear Ham tau}, i.e.,
\begin{equation}
\begin{cases}
\hat{\ga}^\prime(\tau,\tau_{\vep})&=J\hat{B}(\tau)\hat{\ga}(\tau,\tau_{\vep}), \\
\hat{\ga}(\tau_{\vep},\tau_{\vep})&=I.
\end{cases} \nonumber
\end{equation}

By the monotone property of the Maslov index
\bea
\begin{cases}
\mu(\vd,\hat{\gm}(\tau,\tau_{\vep})\vd;[\tau_{\vep},\tau_{2}]) \leq
\mu(\vd,\hat{\gamma}^{+}(\tau,\tau_{\vep})\vd;[\tau_{\vep},\tau_{2}]);\\
\mu(\vd,\hat{\gm}(\tau,\tau_{\vep})\vd;[\tau_{\vep},\tau_{2}]) \geq
\mu(\vd,\hat{\gamma}^{-}(\tau,\tau_{\vep})\vd;[\tau_{\vep},\tau_{2}]).
\end{cases}
\label{48}
\eea
where $\hat{\gamma}^{\pm}(\tau,\tau_{\vep})$ satisfies
\begin{equation}
\begin{cases}
\hat{\ga}^{\pm^{\prime}}(\tau,\tau_{\vep})&=J(\hat{B}_{1}(\tau^{*})\diamond\hat{B}_{2}(\tau^{*})\pm\varepsilon I_{2}
 \diamond E_{\hat{M}_*})\hat{\ga}^\pm(\tau,\tau_{\vep}); \\
\hat{\ga}^{\pm}(\tau_{\vep},\tau_{\vep})&=I.
\end{cases} \nonumber
\end{equation}
Meanwhile by the decomposition property of the Maslov index,
\bea
\mu(\vd,\hat{\gamma}^{\pm}(\tau,\tau_{\vep})\vd;[\tau_{\vep},\tau_{2}])
=\sum_{i=1}^{2}\mu(\vd,\hat{\gamma}^{\pm}_{i}(\tau,\tau_{\vep})\vd;[\tau_{\vep},\tau_{2}]),\label{49}
\eea
where $\hat{\gamma}^{\pm}_{1}(\tau,\tau_{\vep}),\hat{\gamma}^{\pm}_{2}(\tau,\tau_{\vep})$ satisfy
\begin{equation}
\begin{cases}
\hat{\ga}^{\pm^{\prime}}_{1}(\tau,\tau_{\vep})&=J(\hat{B}_{1}(\tau^{*})\pm\varepsilon I)
\hat{\ga}^{\pm}_{1}(\tau,\tau_{\vep});\\
\hat{\ga}^{\pm}_{1}(\tau_{\vep},\tau_{\vep})&=I,
\end{cases}
\nonumber
\end{equation}
\begin{equation*}
\begin{cases}
\hat{\ga}^{\pm^{\prime}}_{2}(\tau,\tau_{\vep})&=J(\hat{B}_{2}(\tau^{*})\pm\varepsilon E_{\hat{M}_{*}})\hat{\ga}^{\pm}_{2}(\tau,\tau_{\vep}); \\
\hat{\ga}^{\pm}_{2}(\tau_{\vep},\tau_{\vep})&=I.
\end{cases}
\end{equation*}

Let $\hat{\ga}^{\pm}_{1}(\tau,\tau_{\vep})\vd=(b^{\pm}(\tau),d^{\pm}(\tau))^{T}$, then
\bea
\mu(\vd,\hat{\ga}^{\pm}_{1}(\tau,\tau_{\vep})\vd;[\tau_{\vep},\tau_{2}])=
\#\{\tau: d^{\pm}(\tau)=0,\tau\in[\tau_{\vep},\tau_{2}]\}.\nonumber
\eea
Direct computations show that $d(\tau)$ satisfy
\begin{equation}
\begin{cases}
d^{\pm^{\prime\prime}}(\tau)&=g^{\pm}(\varepsilon)d^{\pm}(\tau), \\
d^{\pm^{\prime}}(\tau_{\vep})&=1\pm\varepsilon,\\
d^{\pm}(\tau_{\vep})&=0.
\end{cases} \nonumber
\end{equation}
where
$$ g^{\pm}(\varepsilon)=\frac{9}{16}v^{*^{2}}+(1\mp\varepsilon)(2U(s^{+})\pm\varepsilon),$$
and $g(\varepsilon)>0$, when $0 < \vep < \min \{1, 2U(s^*)\}$. Hence
\bea
\mu(\vd,\hat{\ga}^{\pm}_{1}(\tau,\tau_{\vep})\vd;[\tau_{\vep},\tau_{2}])=1.\label{47}
\eea

For $\hat{\ga}^{\pm}_{2}(\tau,\tau_{\vep})$, there exist two matrices $A$ and $C$ satisfying
$$ A^{T}\hat{M}_*A=I, C^{T}C=I, $$
and
$$ C^{T}(A^{T}\hat{U}_{xx}(x^{*})A)C=\text{diag}(\lmd_{1}(s^*),\cdots,\lmd_{n^{*}-1}(s^*))). $$
Set $ A_{d}=\text{diag}(A^{T},A^{-1})$ and $C_{d}=\text{diag}(C^{T},C^{-1})$. Then
$$\hat{\xi}^{\pm}(\tau,\tau_{\vep})=C_{d}A_{d}\hat{\gm}^{\pm}_{2}(\tau,\tau_{\vep})A_{d}^{-1}C_{d}^{-1} $$
satisfies equations,
\begin{equation}
\begin{cases}
\hat{\xi}^{\pm^{\prime}}(\tau,\tau_{\vep})&=J\Phi_{C_{d}A_{d}}(\hat{B}_{2}(\tau^{*})\pm\varepsilon E_{\hat{M}_*})\hat{\xi}^\pm(\tau,\tau_{\vep}), \\
\hat{\xi}^{\pm}(\tau_{\vep},\tau_{\vep})&=I.
\end{cases} \nonumber
\end{equation}
where
\begin{equation*}
\begin{aligned}
\Phi_{C_{d}A_{d}}(\hat{B}_{2}(\tau^{*})\pm\varepsilon E_{\hat{M}_*})= & \left(
                                           \begin{array}{cc}
                                             1\pm\varepsilon & \frac{1}{4}v^{*} \\
                                             \frac{1}{4}v^{*} & -\lmd_{1}(s^*)\pm\varepsilon \\
                                           \end{array}
                                         \right)
\diamond\cdots \\
& \dots \diamond\left(
                \begin{array}{cc}
                  1\pm\varepsilon & \frac{1}{4}v^{*} \\
                  \frac{1}{4}v^{*} & -\lmd_{n^{*}-1}(s^*)\pm\varepsilon \\
                \end{array}
              \right)
\end{aligned}
\end{equation*}

A similar argument as in the proof of Lemma \ref{lem: Maslov index t tau} shows
\bea
\mu(\vd,\hat{\gamma}^{\pm}_{2}(\tau,\tau_{\vep})\vd;[\tau_{\vep},\tau_{2}])=\mu(\vd,\hat{\xi}^{\pm}(\tau,\tau_{\vep})\vd;[\tau_{\vep},\tau_{2}]).
\label{50}
\eea
Again by the decomposition property of the Maslov index,
\bea
\mu(\vd,\hat{\xi}^{\pm}(\tau,\tau_{\vep})\vd;[\tau_{\vep},\tau_{2}])=
\sum_{i=1}^{n^{*}-1}\mu(\vd,\hat{\xi}^{\pm}_{i}(\tau,\tau_{\vep})\vd;[\tau_{\vep},\tau_{2}]),\label{51}
\eea
where each $\hat{\xi}^{\pm}_{i}(\tau,\tau_{\vep})$ satisfies
\begin{equation}
\begin{cases}
\hat{\xi}^{*^{\prime}}_{i}(\tau,\tau_{\vep})&=J\left(
                                              \begin{array}{cc}
                                                1\pm\varepsilon & \frac{1}{4}v^{*} \\
                                                \frac{1}{4}v^{*} & -\lmd_{i}(s^*)\pm\varepsilon \\
                                              \end{array}
                                            \right)
\hat{\xi}^{\pm}_{i}(\tau,\tau_{\vep}), \\
\hat{\xi}^{\pm}_{i}(\tau_{\vep},\tau_{\vep})&=I.
\end{cases} \nonumber
\end{equation}

Let $\hat{\xi}^{\pm}_{i}(\tau,\tau_{\vep})\vd=(a^{\pm}_{i}(\tau),c^{\pm}_{i}(\tau))^{T}
$. Then
\bea
\mu(\vd,\hat{\xi}^{\pm}_{i}(\tau,\tau_{\vep})\vd;[\tau_{\vep},\tau_{2}])=
\#\{\tau: c^{\pm}_{i}(\tau)=0,\tau\in[\tau_{\vep},\tau_{2}]\}.\label{52}
\eea
By a direct computation, $c_i^{\pm}(\tau)$ satisfies
\begin{equation} \label{eq: ci}
\begin{cases}
c^{\pm^{\prime\prime}}_{i}(\tau)&=f^{\pm}_{i}(\varepsilon)c^{\pm}_{i}(\tau), \\
c^{\pm^{\prime}}_{i}(\tau_{\vep})&=1\pm\varepsilon,\\
c_{i}(\tau_{\vep})&=0,
\end{cases}
\end{equation}
where
\bea f^{\pm}_{i}(\varepsilon)=\frac{1}{16}v^{*^{2}}+(1\mp\varepsilon)(\lmd_{i}\pm\varepsilon). \label{53}\eea
Assume
$$ l=\#\{1 \le i \le n^*-1:\lmd_{i}(s^*)<\frac{-U(s^*)}{8}\}, $$
$$ w=\#\{1 \le i \le n^*-1:\lmd_{i}(s^*)=\frac{-U(s^*)}{8}\}. $$
Let $\vep$ be smaller than
$$   -\frac{\lmd_{l}(s^*)}{2}+\frac{1}{2}-\frac{1}{2}\sqrt{(\lmd_{l}(s^*)+1)^{2}+U(s^*)/2}$$
and, if $l+w+1 \le n^*-1$, also smaller than
$$ \frac{\lmd_{l+w+1}(s^*)}{2}-\frac{1}{2}+\frac{1}{2}\sqrt{(\lmd_{l+w+1}(s^*)+1)^{2}+U(s^*)/2}. $$
Then
$$
\begin{cases}
f^{+}_{i}(\varepsilon) <0, \ \ \text{ if }  1\leq i\leq l+w,\\
f^{+}_{i}(\varepsilon) >0, \ \ \text{ if } l+w < i \le n^*-1,
\end{cases}
$$
and
$$ \begin{cases}
f^{-}_{i}(\varepsilon) <0, \ \ \text{ if } 1\leq i\leq l,\\
f^{-}_{i}(\varepsilon) >0, \ \ \text{ if } l < i \le n^*-1.
\end{cases}
$$
Then by \eqref{eq: ci}, we get
\begin{equation*}
c^{+}_{i}(\tau)=
\begin{cases}
\frac{1+\varepsilon}{\sqrt{-f^{+}_{i}(\varepsilon)}} \sin(\sqrt{-f^{+}_{i}(\varepsilon)}(\tau-\tau_{\vep})), & \text{ if } 1\leq i\leq l+w;\\
 1, & \text{ if } l+w < i \le n^*-1,
\end{cases}
\end{equation*}

\begin{equation*}
c^{-}_{i}(\tau)=
\begin{cases}\frac{1-\varepsilon}{\sqrt{-f^{-}_{i}(\varepsilon)}} \sin(\sqrt{-f^{-}_{i}(\varepsilon)}(\tau-\tau_{\vep})), & \text{ if } 1\leq i\leq l;\\
 1, & \text{ if } l < i \le n^*-1.
\end{cases}
\end{equation*}
These imply
\begin{equation*}
\#\{\tau: c^{+}_{i}(\tau)=0,\tau\in[\tau_{\vep},\tau_{2}]\}=
\begin{cases}
[\frac{\sqrt{-f^{+}_{i}(\varepsilon)}(\tau_{2}-\tau_{\vep})}{\pi}]+1, & \text{ if } 1\leq i\leq l+w;\\
1, & \text{ if } l+w < i \le n^*-1,
\end{cases}
\end{equation*}
\begin{equation*}
\#\{\tau: c^{-}_{i}(\tau)=0,\tau\in[\tau_{\vep},\tau_{2}]\}=
\begin{cases}
[\frac{\sqrt{-f^{-}_{i}(\varepsilon)}(\tau_{2}-\tau_{\vep})}{\pi}]+1, & \text{ if } 1\leq i\leq l,\\
1, & \text{ if } l < i \le n^*-1.
\end{cases}
\end{equation*}

Together with (\ref{48}),(\ref{49}),(\ref{47}),(\ref{50}),(\ref{51}),(\ref{52}), we get
\bea
\mu(\vd,\hat{\gm}(\tau,\tau_{\vep})\vd;[\tau_{\vep},\tau_{2}])\leq
\sum_{i=1}^{l+w}[\frac{\sqrt{-f^{+}_{i}(\varepsilon)}(\tau_{2}-\tau_{\vep})}{\pi}]+n^{*}. \nonumber\\
\mu(\vd,\hat{\gm}(\tau,\tau_{\vep})\vd;[\tau_{\vep},\tau_{2}])\geq
\sum_{i=1}^{l}[\frac{\sqrt{-f^{-}_{i}(\varepsilon)}(\tau_{2}-\tau_{\vep})}{\pi}]+n^{*}.\nonumber
\eea

Let $t_{\vep}= t(\tau_{\vep})$ and $t_2 = t(\tau_2)$ be the Newtonian times corresponding to $\tau_{\vep}$ and $\tau_2$. By the Morse index theorem \eqref{Mor-Mas} and Lemma \ref{lem: Maslov index t tau}, we have
\begin{equation} \label{55}
\begin{aligned}
m^{-}(q;t_{\vep},t_{2}) & \leq\sum_{i=1}^{l+w}[\frac{\sqrt{-f^{+}_{i}(\varepsilon)}(\tau_{2}-\tau_{\vep})}{\pi}], \\
m^{-}(q;t_{\vep},t_{2}) & \geq\sum_{i=1}^{l}[\frac{\sqrt{-f^{-}_{i}(\varepsilon)}(\tau_{2}-\tau_{\vep})}{\pi}],
\end{aligned}
\end{equation}

From the proof of Lemma \ref{lem:masy1}, we get
$$\lim_{t \to T^{+}}\frac{\tau(t)}{|\ln\beta(t)|}=\frac{1}{3\sqrt{2U(s^{+})}}.$$
Combining this with (\ref{55}) gives us
\bea
\overline{\lim}_{t_{2} \to T^{+}}\frac{m^{-}
(q;t_{\vep},t_{2})}{|\ln\beta(t_{2})|}\leq \frac{1}{3\sqrt{2}\pi}\sum_{i=1}^{l+w}\sqrt{\frac{-f^{+}_{i}(\varepsilon)}{U(s^{+})}},\nonumber\\
\underline{\lim}_{t_{2} \to T^{+}}\frac{m^{-}
(q;t_{\vep},t_{2})}{|\ln\beta(t_{2})|}\geq \frac{1}{3\sqrt{2}\pi}\sum_{i=1}^{l}\sqrt{\frac{-f^{-}_{i}(\varepsilon)}{U(s^{+})}}.\nonumber
\eea

Fix an arbitrary $t_1 < t_{\vep}$, combining the above ineqaulities with Lemma \ref{lem3.1} and the fact that $\lim_{t_2 \to T^+} |\beta(t_2)| = +\infty$ gives us
\bea
\overline{\lim}_{t_{2} \to T^{+}}\frac{m^{-}
(q;t_{1},t_{2})}{|\ln\beta(t_{2})|} \leq \overline{\lim}_{t_{2} \to T^{+}}\frac{m^{-}
(q;t_{\vep},t_{2})}{|\ln\beta(t_{2})|}\leq \frac{1}{3\sqrt{2}\pi}\sum_{i=1}^{l+w}\sqrt{\frac{-f^{+}_{i}(\varepsilon)}{U(s^{+})}},\nonumber\\
\underline{\lim}_{t_{2} \to T^{+}}\frac{m^{-}
(q;t_{1},t_{2})}{|\ln\beta(t_{2})|}\geq \underline{\lim}_{t_{2} \to T^{+}}\frac{m^{-}
(q;t_{\vep},t_{2})}{|\ln\beta(t_{2})|}\geq \frac{1}{3\sqrt{2}\pi}\sum_{i=1}^{l}\sqrt{\frac{-f^{-}_{i}(\varepsilon)}{U(s^{+})}}.\nonumber
\eea
Since the above equalities hold for any $\vep>0$ small enough, the desired result follows from the fact that
\begin{equation*}
\lim_{\varepsilon \to 0}f_{i}^{\pm}(\varepsilon)= \begin{cases}
\frac{U(s^{*})}{8}+\lmd_{i}, & \text{ when } 1 \le i \le l ; \\
0, & \text{ when } l+1 \le i \le l+w.
\end{cases}
\end{equation*}

\end{proof}

\subsection{The strict non-spiral case.} In this subsection, we assume $q(T^{\pm})$ is either a hyperbolic infinity, a parabolic infinity or a total collision, where in the last two cases we further assume the corresponding limit $s^{\pm} = \lim_{t \to T^{\pm}} s(t)$ exists and satisfies the strict non-spiral condition. A proof of Property (b) in Theorem \ref{thm 1 singular} will be given in this subsection.

In this case, it is difficult to compute the Maslov index in \eqref{Mor-Mas} directly. Instead we introduce and compute another Maslov index. Then estimate the difference between these two different Maslov indices using the H\"{o}rmander index. This new Maslov index was first introduced in \cite{CH07} to study the index of homoclinic solutions (further works can be found in \cite{HO} and \cite{HP17}).

\begin{defi} \label{dfn: stable subspace}
At each moment $\tau \in \rr$, we define the \textbf{stable subspace} $V^s(\tau)$ and \textbf{unstable subspace} $V^u(\tau)$ associated with the linear system \eqref{eq: linear Ham tau} as
\begin{equation}
\begin{split}
V^s(\tau) &=\{v \in \mathbb{R}^{2n^*}|\lim_{\sg \to +\infty}\hat{\gamma}(\sg, \tau) v=0\}, \\
V^u(\tau) &=\{v \in \mathbb{R}^{2n^*}|\lim_{\sg \to -\infty}\hat{\gamma}(\sg, \tau) v=0\}.
\end{split}
\end{equation}
\end{defi}

To study the limiting behaviors of the stable and unstable subspaces some topology of the linear subspaces will be needed. Let $\mathcal{G}(\rr^{2k})$ be the Grassmannian of $\rr^{2k}$, i.e. the set of all closed linear subspaces of $\rr^{2k}$. For any $W \in \mathcal{G}(\rr^{2k})$, let $P_W$ be the orthogonal projection of $\rr^{2k}$ to $W$. Then
$$ \text{dist}(W, W^*):= \|P_W-P_{W^*}\|, \text{ for any } W, W^* \in \mathcal{G}(\rr^{2k}), $$
defines a complete metric on $\mathcal{G}(\rr^{2k})$. Here $\|\cdot \|$ represents the metric on the space of bounded linear operators from $\rr^{2k}$ to itself.

\begin{defi}
Given an arbitrary hyperbolic matrix $A$, we define $V^{+}(A)$ and $V^{-}(A)$ as the invariant linear subspaces corresponding to the eigenvalues with positive and negative real parts.
\end{defi}

Set $\tau^*= \pm \infty$. By Proposition \ref{prop: hyp type I} and \ref{prop bh hyperbolic type II}, $\bh(\tau^*)= \lim_{\tau \to \tau^*} \bh(\tau)$ exists and $J\bh(\tau^*)$ is always a hyperbolic matrix under our assumption. Next result gives the limits of the linear subspaces introduced above, for a proof see \cite[Theorem 2.1]{AM}.

\begin{prop} \label{prop: Vus limit}
\begin{enumerate}
\item[(a).] If $J\bh(+\infty)$ is a hyperbolic matrix and $W$ is a linear subspace satisfying $W \pitchfork V^s(\tau_1)$ in $\rr^{2n^*}$, for some $\tau_1$, then
$$\lim_{\tau \to +\infty}V^s(\tau)=V^-(J\bh(+\infty)), \;\; \lim_{\tau \to +\infty} \hat{\gm}(\tau, \tau_1)W = V^+(J\bh(+\infty)). $$
\item[(b).] If $J\bh(-\infty)$ is a hyperbolic matrix and $W$ is a linear subspace satisfying $W \pitchfork V^u(\tau_1)$ in $\rr^{2n^*}$, for some $\tau_1$, then
$$\lim_{\tau \to -\infty} V^u(\tau)=V^+(J\bh(-\infty)), \;\; \lim_{\tau \to -\infty} \hat{\gm}(\tau, \tau_1)W = V^-(J\bh(-\infty)). $$
\end{enumerate}
\end{prop}

%Let $\tau^\pm\in\rr\cup\{\pm\infty\}$,  and assume $\lim_{\tau\to \tau^-} V^u(\tau)$ exist, then we define the  index

With the above proposition, we define another Maslov index as following.
\begin{defi}  \label{dfn: maslov index hete}
For any $\tau_0 \in \rr$, we define the \textbf{Maslov index} as
\begin{equation*}
\mu(\bh; \tau_0):= \mu(\vd,V^u(\tau); (-\infty,\tau_0]).
\end{equation*}
and its limit as (if the limit exists)
\begin{equation*}
\mu(\bh; \rr):=\lim_{\tau_0\to\infty}\mu(\bh; \tau_0).
\end{equation*}
We define the \textbf{degenerate index} as
\begin{equation*} \nu(\bh):=\text{dim}(V^u(0)\cap V^s(0)).
\end{equation*}
We say $\bh$ is \textbf{degenerate}, if $\nu(\bh) \ne 0$.
\end{defi}

The index defined above was introduced in the study of heteroclinic orbits (see \cite{HO, HP17}).
 Since we have assumed $\hat{B}(t)|_{\vd}>0$, then all crossings are positive \cite{RS1}.
\begin{equation}
\mu(\vd,V^u(\tau); \rr) =\sum_{\tau \in \rr} \text{dim}(V^u(\tau)\cap \vd).
\label{1.3c.1.1}
\end{equation}
Under the condition that $J\hat{B}(\pm\infty)$ are hyperbolic,   $-J\frac{d}{dt}-\hat{B}$ is a Fredholm operator and
\[\nu(\hat{B})=dim\ker (-J\frac{d}{dt}-\hat{B}), \]
For the details, see \cite{HP17}.

Now we use H\"{o}rmander index to estimate the difference between the two Maslov indices.
\begin{lem}\label{lem: mas rel}
\begin{enumerate}
\item[(a).] For any $\tau_1 < \tau_2$,
\begin{equation*}
\mu(\vd,\hat{\gamma}(\tau,\tau_{1})\vd;[\tau_{1},\tau_{2}])-\mu(\vd,V^{u}(\tau);[\tau_{1},\tau_{2}]) =s(\hat{\gamma}(\tau_{1},\tau_{2})\vd,\vd; V^{u}(\tau_{1}),\vd).
\end{equation*}
\item[(b).] For any $\tau_0 \in \rr$ and $T_0= t(\tau_0)$. If $\vd\pitchfork V^u(\tau_0)$ and $\vd\pitchfork V^{+}(J\hat{B}(-\infty))$,
\begin{equation} \label{eq: two Mas horm tau0}
m^{-}(q; T^-, T_0)+n^{*} -\mu(\hat{B}; \tau_0) = s(V^{-}(J\hat{B}(-\infty)), \vd; V^{+}(J\hat{B}(-\infty)), \vd).
\end{equation}
\item[(c).] If $V^{s}(0)\pitchfork V^{u}(0)$,  $\vd\pitchfork V^{+}(J\hat{B}(+\infty))$ and $\vd\pitchfork V^{+}(J\hat{B}(-\infty))$, then $\mu(\bh; \rr)$ is a finite number and
\begin{equation} \label{eq: two Mas horm}
m^{-}(q; T^-, T^+)+n^{*} -\mu(\hat{B}; \rr) = s(V^{-}(J\hat{B}(-\infty)), \vd; V^{+}(J\hat{B}(-\infty)), \vd).
\end{equation}
\end{enumerate}
\end{lem}

\begin{proof}
(a). Let $\Lambda_{s}, s\in[0,1]$, be a continuous path of Lagrangian subspaces satisfying $\Lambda_{0}=V^{u}(\tau_{1})$ and $\Lambda_{1}=\vd$.
By the homotopy invariant property of the Maslov index,
\begin{align*}
\mu(\vd,\Lambda_{s};[0,1]) & +\mu(\vd,\hat{\gamma}(\tau,\tau_{1})\vd;[\tau_{1},\tau_{2}])\\
 &=\mu(\vd,V^{u}(\tau);[\tau_{1},\tau_{2}])+\mu(\vd,\hat{\gamma}(\tau_{2},\tau_{1})\Lambda_{s};[0,1]).
\end{align*}
Then
\begin{equation} \label{eq: hormand1}
\begin{split}
\mu(\vd,\hat{\gamma}(\tau,\tau_{1})\vd; & \; [\tau_{1},\tau_{2}])-\mu(\vd,V^{u}(\tau);[\tau_{1},\tau_{2}])\\
&=\mu(\vd,\hat{\gamma}(\tau_{2},\tau_{1})\Lambda_{s};[0,1])-\mu(\vd,\Lambda_{s};[0,1])\\
&=\mu(\hat{\gamma}(\tau_{1},\tau_{2})\vd,\Lambda_{s};[0,1])-\mu(\vd,\Lambda_{s};[0,1])\\
&=s(\hat{\gamma}(\tau_{1},\tau_{2})\vd,\vd; V^{u}(\tau_{1}),\vd).
\end{split}
\end{equation}

(b). For any $\tau_- < \tau_0$. Let $t_-= t(\tau_-)$. By \eqref{Mor-Mas} and Lemma \ref{lem: Maslov index t tau},
$$ m^-(q; t_-, T_0)+ n^* = \mu(\vd,\hat{\gamma}(\tau,\tau_{-})\vd; \; [\tau_-, \tau_{0}]). $$
Then Proposition \ref{lem: Maslov index t tau} and the monotone property of Morse index imply
\begin{equation}
\label{eq: mor mas lim tau0}
 m^-(q; T^-, T_0) + n^* =  \lim_{\tau_- \to -\infty}  \mu(\vd,\hat{\gamma}(\tau,\tau_{-})\vd; \; [\tau_{-},\tau_{0}]).
\end{equation}
Since $\vd\pitchfork V^u(\tau_0)$, then $\lim_{\tau\to-\infty}\hat{\gamma}(\tau,\tau_0)\vd=V^-(J\hat{B}(-\infty)) $.
Then
\begin{equation} \label{eq: mas typ lim tau0}
\begin{split}
 \lim_{\tau \to -\infty} & s(\hat{\gm}(\tau, \tau_0)\vd, \vd; V^u(\tau), \vd) \\
& = s(V^{-}(J\hat{B}(-\infty)),\vd; V^{+}(J\hat{B}(-\infty)),\vd).
\end{split}
\end{equation}
Meanwhile $\vd\pitchfork V^+(J\hat{B}(-\infty))$ implies
\begin{equation}
\label{hatbtau0} \mu(\hat{B}; \tau_0) = \lim_{\tau_- \to -\infty}  \mu(\vd, V^u(\tau); \; [\tau_{-},\tau_{0}]).
\end{equation}
\eqref{eq: two Mas horm tau0} now follows from \eqref{eq: mor mas lim tau0}, \eqref{eq: mas typ lim tau0} and \eqref{hatbtau0}.

(c). Since $V^{s}(0)\pitchfork V^{u}(0)$, by Proposition \ref{prop: Vus limit},
\begin{equation}
\label{eq: lim Vu} \lim_{\tau \to -\infty} V^u(\tau) = V^+(J\bh(-\infty)), \;\; \lim_{\tau \to +\infty} V^u(\tau) = V^+(J\bh(+\infty)).
\end{equation}
This means there exist $\tau_1 <0$ small enough and $\tau_2 >0 $ large enough, such that
$$ \vd \pitchfork V^u(\tau), \;\; \forall \tau \in (-\infty, \tau_1] \cup [\tau_2, +\infty). $$
Hence for any $\tau'_1 \le \tau_1$ and $\tau'_2 \ge \tau_2$, $\mu(\vd,V^{u}(\tau);[\tau'_{1},\tau'_{2}])$ is a finite constant and
\begin{equation}
\label{eq: mu bh finite} \mu(\bh; \rr)=\mu(\vd,V^{u}(\tau);\mathbb{R})=\mu(\vd,V^{u}(\tau);[\tau'_{1},\tau'_{2}]).
\end{equation}
This implies $\mu(\bh; \rr)$ must be finite.

Let $t_i= t(\tau_i)$, $i =1, 2$. By \eqref{Mor-Mas} and Lemma \ref{lem: Maslov index t tau},
$$ m^-(q; t_1, t_2)+ n^* = \mu(\vd,\hat{\gamma}(\tau,\tau_{1})\vd; \; [\tau_{1},\tau_{2}]). $$
Then Proposition \ref{lem: Maslov index t tau} and the monotone property of Morse index imply
\begin{equation}
\label{eq: mor mas lim}
 m^-(q; T^-, T^+) + n^* = \lim_{\tau_2 \to +\infty} \lim_{\tau_1 \to -\infty}  \mu(\vd,\hat{\gamma}(\tau,\tau_{1})\vd; \; [\tau_{1},\tau_{2}]).
\end{equation}

Since $\vd \pitchfork V^+(J\bh(+\infty))$, the second identity in \eqref{eq: lim Vu} implies $\vd \pitchfork V^u(\tau'_2)$, $ \forall \tau'_2 \ge \tau_2$, when $\tau_2$ is large enough. Then by Proposition \ref{prop: Vus limit},
$$ \lim_{\tau_1 \to -\infty} \hat{\gm}(\tau_1, \tau'_2) \vd = V^-(J\bh(-\infty)), \;\; \forall \tau'_2 \ge \tau_2. $$
As a result, for $\tau_2$ large enough,
\begin{equation} \label{eq: mas typ lim}
\begin{split}
\lim_{\tau_2 \to +\infty} \lim_{\tau_1 \to -\infty} & s(\hat{\gm}(\tau_1, \tau_2)\vd, \vd; V^u(\tau_1), \vd) \\
& = s(V^{-}(J\hat{B}(-\infty)),\vd; V^{+}(J\hat{B}(-\infty)),\vd)
\end{split}
\end{equation}
Combining \eqref{eq: hormand1}, \eqref{eq: mu bh finite}, \eqref{eq: mor mas lim} and \eqref{eq: mas typ lim}, we get the desired identity.
\end{proof}

The identity \eqref{eq: two Mas horm} gives the difference between the two different Maslov indices as the H\"{o}rmander index. The next lemma will be useful in computing the H\"{o}rmander index.

\begin{lem}\label{lem:Homand}
Given a matrix $B= \text{diag}(1, b)$ with $b<0$.
\begin{equation*}
V^{\pm}(JB)= \left\langle \left(
     \begin{array}{c}
       \pm\sqrt{-b} \\
       1 \\
     \end{array}
   \right) \right\rangle \pitchfork \vd
\end{equation*}
and
\begin{equation*}
s \left(V^{-}(JB), \vd; V^{+}(JB),  \vd \right) =1.
\end{equation*}
\end{lem}

\begin{proof} Let $\Lambda_t=\left\langle \left(
     \begin{array}{c}
       \sqrt{-b} \\
       1-t \\
     \end{array}
   \right) \right\rangle $,  $t\in[0,1]$, then $\Lambda_0=V^+(JB)  $,  $\Lambda_1=\vd$.
 The result follows from the facts that $\mu(V^-(JB),\Lambda_t )=0$ and $\mu(\Lambda_1,\Lambda_t )=-1$.

\end{proof}

%\begin{lem}
%\label{lem: Hormand tau*} If $q(T^{\pm})$ is a hyperbolic singularity or a collision/parabolic singularity with the corresponding $s^{\pm}$ satisfying the strict non-spiral condition, then
%$$ s(V^{-}(J\hat{B}(\tau^*)), \vd; \; V^{+}(J\hat{B}(\tau^*)), \vd) =n^*. $$
%\end{lem}

Now we are ready to prove property (b) in Theorem \ref{thm 1 singular}, which follows directly from next theorem.
\begin{thm}
\label{thm: hyper}
Let $q \in C^2(T^-, T^+), \hat{\X})$ be a doubly asymptotic solution satisfying the conditions given in property (b) of Theorem \ref{thm 1 singular}.
\begin{enumerate}
\item[(a).] For any $\tau_0 \in \rr$, let $T_0=t(\tau_0)$ be the corresponding Newtonian time. If $\vd\pitchfork V^u(\tau_0)$,    then
\begin{equation} \label{eq: m-=mutau0}
m^-(q; T^-, T_0)=\mu(\hat{B}; \tau_0).
\end{equation}
Moreover if $V^{s}(0)\pitchfork V^{u}(0)$, then
\begin{equation} \label{eq: m-=mu}
m^-(q; T^-, T^+)=\mu(\hat{B}; \rr).
\end{equation}
\item[(b).] $m^-(q; T^-, T^+)$ is finite.
\end{enumerate}
\end{thm}

\begin{proof}
(a). We will prove \eqref{eq: m-=mutau0} first. Depending the property of $q(T^-)$, different types McGehee coordinates need to be used. First let's assume $q(T^-)$ is a total collision or a parabolic infinity. Set $\tau^*=-\infty$. Following the notations from Section \ref{sec: McGehee coordinates}, $s^* = \lim_{\tau \to \tau^*}s(t)$ is a central configuration, $x^* = \psi(s^*)$, and $\hat{B}(\tau^{*})=\hat{B}_{1}(\tau^*)\diamond\hat{B}_{2}(\tau^*)$ with $\hat{B}_{1}(\tau^*), \hat{B}_{2}(\tau^*)$ given as in \eqref{eq: B hat 1 2 type I}. Recall that Proposition \ref{prop: hyp type I} shows $J\hat{B}(\tau^*)$ is a hyperbolic matrix.

By Lemma \ref{lem: eigenvalue U hat}, there exists a matrix $A$, such that $A^{T}\Mh_*A=I$ and the eigenvalues of $A^T \U_{xx}(x^*) A$ are $\lmd_i(s^*)$, $i =1, \cdots, n^*-1$. We can further find an orthogonal matrix $C$, such that
$$ C^{T}(A^{T}\hat{U}_{xx}(x^*)A)C=\text{diag}(\lmd_{1}(s^*),\cdots,\lmd_{n^{*}-1}(s^*)). $$
Set $R=R_{1}\diamond (R_{3}R_{2})$ with
$$ R_{1}=\left(
         \begin{array}{cc}
           1 & -\frac{3}{4}v^{*} \\
           0 & 1 \\
         \end{array}
       \right), \;\;
R_{2}= \left(
\begin{array}{cc}
           C^{T}A^{T} & 0 \\
           0 & C^{-1}A^{-1}\\
\end{array}
\right),
$$
$$
R_{3}=\left(
            \begin{array}{cc}
                1 & \frac{1}{4}v^{*} \\
                0 & 1 \\
                \end{array}
                \right)\diamond\cdots\diamond
                \left(
                      \begin{array}{cc}
                            1 & \frac{1}{4}v^{*} \\
                            0 & 1 \\
                            \end{array}
                                       \right),
                $$
where $R_2$ and $R_3$ are $2(n^*-1) \times 2(n^*-1)$ matrices. According to \eqref{eq: phi R B},
\begin{align*}
\Phi_{R}(\hat{B})(\tau^{*}) =\left(
                              \begin{array}{cc}
                                1 & 0 \\
                                0 & -\frac{25}{8}U(s^{*}) \\
                              \end{array}
                            \right) & \diamond
                            \left(
                              \begin{array}{cc}
                                1 & 0 \\
                                0 & -\frac{U(s^*)}{8}-\lmd_{1}(s^*) \\
                              \end{array}
                            \right)\diamond \\
                            & \cdots\diamond
                            \left(
                              \begin{array}{cc}
                                1 & 0 \\
                                0 & -\frac{U(s^*)}{8}-\lmd_{n^*-1}(s^*)\\
                              \end{array}
                            \right).
\end{align*}

Recall that the strict non-spiral condition implies
$$ -\frac{U(s^*)}{8}-\lmd_{i}(s^*)<0, \;\; 1\leq i\leq n^{*}-1. $$
Combining these with Lemma \ref{lem:Homand}, we get
\begin{equation}
s(V^{-}(J\Phi_{R}(\hat{B})(\tau^*)), \vd, V^{+}(J\Phi_{R}(\hat{B})(\tau^*)), \vd)=n^{*}.
\end{equation}
Direct computations show
$$ R\vd=\vd, \;\; RV^{\pm}(J\bh(\tau^*)) = V^{\pm}(J \Phi_R(\bh)(\tau^*)). $$
Therefore
\begin{align*}
s(V^{-}(J\hat{B}(\tau^*)), \vd; & \; V^{+}(J\hat{B}(\tau^*)), \vd)\\
			  & =s(RV^{-}(J\hat{B}(\tau^*)), R\vd; RV^{+}(J\hat{B}(\tau^*)), R\vd)\\
              & =s(V^{-}(J\Phi_{R}(\hat{B})(\tau^*)), \vd; V^{+}(J\Phi_{R}(\hat{B})(\tau^*)), \vd)=n^{*}
\end{align*}
Together with   \eqref{eq: two Mas horm tau0}, they imply \eqref{eq: m-=mutau0}.

To prove \eqref{eq: m-=mu},  let's assume $V^{s}(0)\pitchfork V^{u}(0)$, combining with \eqref{eq: two Mas horm}, they imply the desired result.

The proof is exact the same, when $q(T^-)$ is a hyperbolic infinity. The only difference is in this case $\hat{B}_{1}(\tau^*), \hat{B}_{2}(\tau^*)$ are given by (\ref{eq: bh 1 2 type II}), and
$$ R_{1}=\left(
         \begin{array}{cc}
           1 & -\frac{1}{2}v^{*} \\
           0 & 1 \\
         \end{array}
       \right), \;
       R_{2}=\left(
         \begin{array}{cc}
           A^T & 0 \\
           0 & A^{-1} \\
         \end{array}
       \right), \;
                              R_{3}=\left(
                                                   \begin{array}{cc}
                                                     1 & \frac{1}{2}v^{*} \\
                                                     0 & 1 \\
                                                   \end{array}
                                                 \right)\diamond\cdots\diamond
                                                 \left(
                                                   \begin{array}{cc}
                                                     1 & \frac{1}{2}v^{*} \\
                                                     0 & 1 \\
                                                   \end{array}
                                                 \right).
                                                 $$
This finishes our proof of property (a).

(b). Now let's drop the assume that $V^{s}(0)\pitchfork V^{u}(0)$. In this case, we can always find another Lagrange subspace $V$ satisfying $V\pitchfork V^{u}(0)$ and $V\pitchfork V^{s}(0)$. By Proposition \ref{prop: Vus limit}, this implies
\begin{equation} \label{eq: lim gm V}
\lim_{\tau\rightarrow+\infty}\hat{\gamma}(\tau,0)V=V^{+}(J\hat{B}(+\infty)), \lim_{\tau\rightarrow-\infty}\hat{\gamma}(\tau,0)V=V^{-}(J\hat{B}(-\infty)).
\end{equation}
After replacing $V^u(\tau)$ by $\hat{\gm}(\tau, 0)V$, a similar argument as in the proof of property (a) of Lemma \ref{lem: mas rel} shows, for any $\tau_1 < \tau_2$,
\begin{equation*}
\begin{split}
\mu(\vd,\hat{\gamma}(\tau, \tau_{1})\vd; & [\tau_{1},\tau_{2}]) \\
& = \mu(\vd,\hat{\gamma}(\tau,0)V;[\tau_{1},\tau_{2}]) + s(\hat{\gamma}(\tau_{1},\tau_{2})\vd,\vd; \hat{\gamma}(\tau_{1},0)V,\vd).
\end{split}
\end{equation*}

Meanwhile $\vd\pitchfork V^{+}(J\hat{B}(+\infty)), \vd\pitchfork V^{-}(J\hat{B}(-\infty))$. With \eqref{eq: lim gm V}, they imply
$$ \lim_{\tau_2 \to +\infty} \lim_{\tau_1 \to \infty} \mu(\vd,\hat{\gamma}(\tau,0)V;[\tau_{1},\tau_{2}]) < +\infty. $$
Since the H\"{o}rmander index $s(\hat{\gamma}(\tau_{1},\tau_{2})\vd,\vd; \hat{\gamma}(\tau_{1},0)V,\vd)$ is also bounded \eqref{hormd},   together with the monotonicity of $\mu(\vd,\hat{\gamma}(\tau,\tau_{1})\vd;[\tau_{1},\tau_{2}])$, they imply
$$  \lim_{\tau_2 \to +\infty} \lim_{\tau_1 \to \infty} \mu(\vd,\hat{\gamma}(\tau,\tau_{1})\vd;[\tau_{1},\tau_{2}]) < +\infty. $$
Then the finiteness of $m^-(q; T^-, T^+)$ follows from \eqref{eq: mor mas lim}.
\end{proof}

\section{the Morse indices of homothetic solutions} \label{sec: homo index}

In this section, we give a proof of Theorem \ref{thm 2 homo}. Throughout the entire section, we assume $q \in C^2((T^-, T^+), \hat{\X})$ is a homothetic solution with energy $H_0$. As a result, there is a normalized central configuration $s_0$, such that
$$ s(t) = q(t)/ \sqrt{\I(q(t))} \equiv s_0, \forall t \in (T^-, T^+).$$

When $H_0 <0$, the homothetic solution starts from a total collision at a finite time and comes back to another total collision at a different time. When $H_0 \ge 0$, the homothetic solution either starts from a total collision at a finite time and goes to infinity, as time goes to positive infinite, or it ends at the total collision at a finite time and goes to infinity, as time goes to negative infinity. However a simple time reversal changes one to the other. Without loss of generality, in the rest of the section we assume $q(T^-)$ is a total collision with a finite $T^- \in \rr$. Then $T^+ \in \rr$ is finite, when $H_0 <0$ and $T^+= +\infty$, when $H_0 \ge 0$.

To compute the index of $q(t)$, we rewrite it in McGehee coordinates $(v, u, r, x)(\tau)$ with time parameter $\tau$ as defined in subsection \ref{subsec: McGehee coordinates}. Although when the energy $H_0$ is positive, we have a hyperbolic infinity, as $t$ goes to $T^+=+\infty$, we will not switch to the hyperbolic McGehee coordinates, but continue to use the McGehee coordinate.
Because of this, when $H_0>0$, the corresponding time $\tau=\tau(t)$ goes to some finite $\tau^+$, when $t$ goes to $T^+= +\infty$.

Since $q(t)$ is a homothetic solution, $x(\tau) \equiv x_0 = \psi(s_0)$ and $u(\tau) \equiv 0$, for all $\tau$. Then the second and fourth equations in \eqref{eq: McGehee 1.1} are always zero on both sides, and the first and third equations become
\begin{equation}
\begin{cases}
v'&=\frac{1}{2}v^2-b, \\
r'&=rv
\end{cases} \label{eq: McGehee 2}
\end{equation}
where $b = \U(x_0)$ through out this section. The energy identity \eqref{eq: energy M1} now reads
\begin{equation}
\label{eq: energy homothetic} \frac{1}{2}v^2-b=rH_0.
\end{equation}

Moreover \eqref{eq: hat B tau} now becomes $\hat{B}(\tau)=\hat{B}_1(\tau)\diamond \hat{B}_2(\tau) $ with
\begin{equation}
\label{eq: B hat 1 2} \hat{B}_1(\tau)=\left(\begin{array}{cc}1 & -\frac{3}{4}v\\
                             -\frac{3}{4}v&-2b\end{array}\right), \,   \hat{B}_2(\tau)=\left(\begin{array}{cc}\hat{M}^{-1} & \frac{1}{4}vI\\
                             \frac{1}{4}vI&-\U_{xx}(x_0)\end{array}\right).
\end{equation}
Here and in the rest of the section, we set
$$ \hat{M}= \hat{M}(x_0)= \left. \left( \frac{\partial \psi^{-1}}{\partial x} \right) \right|_{x= x_0}^T M \left. \left( \frac{\partial \psi^{-1}}{\partial x} \right) \right|_{x= x_0}. $$
Let $V^u_1, V^u_2$ and $V^u$ be the unstable subspaces of $\hat{B}_1, \hat{B}_2$ and $\bh$ respectively according to Definition \ref{dfn: stable subspace}. Then
$V^u=V^u_1\oplus V^u_2$.  By the Symplectic additivity of Maslov index,
we have
\bea  \mu(\hat{B}; \tau_0) =\mu(\hat{B}_1; \tau_0) +\mu(\hat{B}_2; \tau_0), \label{Vutau0}   \eea
When both limits $\lim_{\tau\to+\infty}V^u_i(\tau), i=1,2$, exist.  After passing $\tau_0$ to $+\infty$, we get
\bea  \mu(\hat{B};\rr) =\mu(\hat{B}_1;\rr)+\mu(\hat{B}_2;\rr). \label{Vuinfty}   \eea

Let $R_1(\tau)=\left(\begin{array}{cc}1 & -\frac{3}{4}v(\tau)\\
0&1\end{array}\right)$.
By Lemma \ref{lem: phi R B},  $\xi_1(\tau)$ satisfies $\xi'_1(\tau)=J\hat{B}_1(\tau)\xi_1(\tau)$ if and only if $\eta_1(\tau)=R_1(\tau)\xi_1(\tau)$ satisfies
\begin{equation} \label{eq: R1 Bhat1}
 \eta'_1(\tau) = J \Phi_{R_1}(\bh_1)(\tau) \eta_1(\tau),
\end{equation}
where according to \eqref{eq: phi R B},
\begin{equation}
\label{eq: R1} \Phi_{R_1}(\hat{B}_1)=\left(\begin{array}{cc}1 & 0\\
    0& \frac{3}{4}v'-\frac{9}{16}v^2-2b   \end{array}\right) =\left(\begin{array}{cc}1 & 0\\
    0& -\frac{3}{11}v^2 -\frac{11}{4}b   \end{array}\right).
\end{equation}

Next lemma will be useful in computing $\mu(\vd, V^u(\tau), (-\infty,\tau_0])$.

\begin{lem}{\cite[Lemma 3.10]{HO}} \label{lem:nondegen} If $y(\tau)$, $\tau \in \rr$, satisfies
\bea\label{61}
y^{\prime\prime}(\tau)=f(\tau)y(\tau),
\eea
where $f \in C^0(\rr, [0, +\infty))$ satisfies $\int_{\tau_1}^{\tau_2}f(\tau)d\tau>0$ for any $\tau_1 < \tau_2$, the $\lim_{\tau\to\pm\infty}f(\tau)>0$. Then
\begin{enumerate}
\item[(a).] for any $\tau_{1}< \tau_{2}$, there is no non-trivial solution of \eqref{61}, which satisfies $y(\tau_{1})y^{\prime}(\tau_{1})\geq y(\tau_{2})y^{\prime}(\tau_{2})$;
\item[(b).] there is no non-trivial solution of \eqref{61} satisfying
$y(\tau_0)y^{\prime}(\tau_0)=0$ for some $\tau_0\in\mathbb{R}$ and $y(\tau)\rightarrow0, y^{\prime}(\tau)\rightarrow0$ as $\tau\rightarrow-\infty$ (or $\tau\to+\infty$);
\item[(c).] there is no non-trivial solution of \eqref{61}, which is bounded on $\mathbb{R}$.
\end{enumerate}
\end{lem}

\begin{proof} The key idea of the proof is to multiply both sides of \eqref{61} by $y$, and then use integration by parts. For details see \cite[Lemma 3.10]{HO}.
\end{proof}

\begin{prop}\label{prop:B1}
Given a homothetic solution with arbitrary energy. For any $\tau_0 \in \rr$,
$$\mu(\hat{B}_1; \tau_0)=0.$$
\end{prop}
\begin{proof}
Let $V^u(\tau)$, $\tilde{V}^u(\tau)$ be the unstable subspaces of the linear Hamiltonian system associated to $\hat{B}_1$ and  $\Phi_{R_1}(\hat{B}_1)$ respectively.
By  Lemma \ref{lem: phi R B},
$ \tilde{V}^u(\tau)=R_1V^u(\tau)  $.   Since $R_1\vd=\vd$, the symplectic invariant property of Maslov index implies
\bea  \mu(\hat{B}_1; \tau_0)=\mu(\vd, V^u(\tau);(\-\infty, \tau_0]) =\mu(\vd, \tilde{V}^u(\tau); (\-\infty, \tau_0]).  \eea
Hence it is enough to show $\mu(\vd, \tilde{V}^u(\tau); (\-\infty, \tau_0])=0$.

Since $\tilde{V}^u(\tau)$ is a 1-dim space. There exist $\xi(\tau) =(y(\tau), x(\tau))^T$, such that $\tilde{V}^u= \langle \xi(\tau) \rangle$. Then $\xi(\tau)$ satisfies \eqref{eq: R1 Bhat1} and $\lim_{\tau \to -\infty} \xi(\tau)=0$, which implies
$$ x''(\tau)=(\frac{3}{11}v^2(\tau) +\frac{11}{4}b)x(\tau) \; \text{ and } \; \lim_{\tau \to -\infty} x(\tau)=0. $$
Since $\frac{3}{11}v^2 +\frac{11}{4}b>\frac{11}{4}b$, from property (b) of Lemma \ref{lem:nondegen}, we have $x(\tau) \ne 0$, $ \forall \tau\in\rr$. This implies that there is no crossing occur for $\vd$, which gives us the desired result.
\end{proof}

Now we show how to compute $\mu(\hat{B}_2,\rr)$.
Recall that we can find a matrix $A$ such that $A^T\hat{M}A=I$ and
  \[\Phi_{A_d}(\hat{B}_{2}(\tau) ) = \left(\begin{array}{cc}I & \frac{1}{4}vI\\
                             \frac{1}{4}vI&-A^T\U_{xx}(x_0)A\end{array}\right), \]
where  $A_d=diag(A^{-T},A)$. Let $\lmd_i$, $i=1, \dots, n^*-1$, represent the eigenvalues of  $M^{-1}D^2U|_{\mathcal{E}}(s_0)$. By Lemma \ref{lem: eigenvalue U hat}, they are the eigenvalues of $A^T\U_{xx}(x_0) A$ as well.
Then after a change of basis, we have
\[ \left(\begin{array}{cc}I & \frac{1}{4}vI\\
                             \frac{1}{4}vI&-A^T\U_{xx}(x_0)A\end{array}\right)=\left(\begin{array}{cc}1 & \frac{1}{4}v\\
                             \frac{1}{4}v&-\lmd_1\end{array}\right)\diamond\cdots\diamond \left(\begin{array}{cc}1 & \frac{1}{4}v\\
                             \frac{1}{4}v&-\lmd_{n^*-1}\end{array}\right). \]
By the  symplectic additivity of Maslov index,  we have
\begin{equation}
\label{eq: MasSum B2}  \mu(\hat{B}_2; \tau_0 ) = \sum_{i =1}^{n^*-1} \mu(\mathcal{B}_{\lambda_i};  \tau_0).
\end{equation}

To simplify notation, let $\lmd$ represent any $\lmd_i$, $i =1, \dots, n^*-1$, and $V^u_\lambda(\tau)$ the unstable subspaces of the linear Hamiltonian system
\begin{equation}
\label{eq: B hat lmd} \eta'(\tau) = J \mathcal{B}_{\lmd}(\tau) \eta(\tau), \; \text{ where }
                             \mathcal{B}_{\lmd}=  \left(\begin{array}{cc}1 & \frac{1}{4}v\\
                             \frac{1}{4}v&-\lmd\end{array}\right).
\end{equation}

%\begin{equation}
%\label{eq: MasSum B2}  \mu(\vd,  V^u(\tau); (-\infty, \tau_0]) = \sum_{i =1}^{n^*-1} \mu(\vd, V^u_{\lambda_i}; (-\infty, \tau_0]).
%\end{equation}

Then for $R_2(\tau)=\left(\begin{array}{cc}1 & \frac{1}{4}v(\tau)\\
0&1\end{array}\right)$,
\begin{equation}
\label{eq: R2lambda} \Phi_{R_2}(\mathcal{B}_\lambda)=\left(\begin{array}{cc}1 & 0\\
    0& -\frac{3}{16}v^2+\frac{b}{4}-\lambda   \end{array}\right) =\left(\begin{array}{cc}1 & 0\\
    0& -\frac{3}{8}rH_0 -\frac{b}{8}-\lambda   \end{array}\right).
\end{equation}
Here we have used the energy relation \eqref{eq: energy homothetic}.
Let $\tilde{V}^u_\lambda(\tau)$ be the unstable subspace of the linear Hamiltonian system
 $\eta'(\tau) = J\Phi_{R_2}(\mathcal{B}_\lambda)\eta(\tau)$. Then $\tilde{V}^u_{\lmd}(\tau)=R_2(\tau) V^u_\lambda(\tau)$. Since $R_2(\tau)\vd=\vd$,
 \bea  \mu(\vd, V^u_\lambda(\tau); (-\infty, \tau_0])= \mu(\vd, \tilde{V}^u_\lambda(\tau); (-\infty, \tau_0]) ,      \eea which is equivariant to $\mu(\mathcal{B}_\lambda; \tau_0)=\mu(\Phi_{R_2}(\mathcal{B}_\lambda); \tau_0)  $.

Since the homothetic solution with different energies behaviors quite differently, first let's assume $H_0 \ge 0$, then $\frac{3}{8}rH_0 +\frac{b}{8}+\lambda>0$. A similar argument as in the proof of Proposition \ref{prop:B1} shows for any $\tau_0 \in \rr$,
 $$\mu(\vd, \tilde{V}^u_\lambda(\tau); (\-\infty, \tau_0])=0.$$
From \eqref{eq: MasSum B2}, we get the following proposition
 \begin{prop} \label{pop:B2h0} For $H_0\geq 0$,  $\mu(\hat{B}_2; \tau_0) =0$.
\end{prop}

Now we will compute the Morse index of $q(t)$.

\begin{prop} \label{prop: Morse PosiZeroEng}
If the energy $H_0 \ge 0$, then $m^-(q;T^-, T^+)=0$.
\end{prop}

\begin{proof}
First let's assume $H_0=0$. Then the energy identity \eqref{eq: energy homothetic} implies $v(\tau) \equiv \sqrt{2b}$, for all $\tau\in\rr$. Notice that $v(\tau)$ is positive, as we assume $q(t)$ goes from the total collision to infinity.  From Proposition \ref{prop:B1} and   Proposition \ref{pop:B2h0}, we have
$$  \mu(\hat{B}; \tau_0) =0, \;\; \forall \tau_0 \in \rr.$$
Since the system is non-degenerate, let $\tau_0\to +\infty$, we get
\bea  \mu(\hat{B}; \rr) =0.\eea
From \eqref{eq: m-=mu},
\bea m^-(q; T^-, T^+)=\mu(\hat{B}; \rr)=0.  \label{morseh0} \eea

Now let's assume $H_0>0$. By our assumption $q(T^-)$ is a total collision with a finite $T^- \in \rr$ and $q(T^+)$ is a hyperbolic infinity with $T^+=+\infty$. Let's fix an arbitrary finite time $T_0 >T^-$, and rewrite $q(t), t \in (T^-, T_0)$ in McGehee coordinates. Then by the energy identity \eqref{eq: energy homothetic},
\begin{equation*}
v(\tau) = \sqrt{2(H_0 r(\tau) +b)}, \; \forall \tau \in (-\infty, \tau_0).
\end{equation*}
Here $\tau_0=\tau(T_0)$ is the moment in McGehee coordinates corresponding to $T_0$. Since $q(T^+)$ is a hyperbolic infinity and we are using McGehee coordinates instead of the hyperbolic McGehee coordinates, there is a finite $\tau^+$, such that $\lim_{T_0 \to T^+} \tau(T_0)=\tau^+$.

By Proposition \ref{prop:B1} and \ref{pop:B2h0},
$$  \mu(\hat{B}; \tau_0) =0, \;\; \forall \tau_0 \in (-\infty, \tau^+). $$
The proofs of Proposition \ref{prop:B1}  and \ref{pop:B2h0} show that $\vd\pitchfork V^u(\tau_0).$
Then \eqref{eq: m-=mutau0} implies
$$m^-(q; T^-, T_0)=\mu(\hat{B}; \tau_0)=0, \; \forall T_0<T^+. $$
Since this is true for any finite $T_0 > T^-$, we have
$$   m^-(q; T^-, T^+) = \lim_{T_0 \to T^+} m^-(q; T^-, T_0)=0. $$
\end{proof}

From now on let's assume the energy $H_0 <0$. Notice that in this case $v(\tau_0)=0$ for some finite $\tau_0$, where $r(\tau)$ reaches its maximum at $\tau =\tau_0$. Without loss of generality, let's assume $\tau_0=0$. Solving the first equation in \eqref{eq: McGehee 2} directly, we get
\begin{equation}
 \label{eq: v tau} v(\tau)=- \sqrt{2b}\tanh \frac{\sqrt{2b}\tau}{2}.
 \end{equation}
As a result,
$$ \bc_{\lmd}(\pm \infty) = \lim_{\tau \to \pm \infty} \bc_{\lmd}(\tau) = \left(\begin{array}{cc}
1 & \mp \sqrt{2b}/4 \\
\mp \sqrt{2b}/4 & -\lmd
\end{array} \right)
$$
\begin{lem}  \label{lem:eigen JB-lambda} For any $\lambda>-b/8$, the eigenvalues of $J\bc_\lambda(\pm\infty)$ are $\pm\sqrt{b/8+\lambda}$, and the corresponding eigenspaces are
\begin{equation}
V^{\pm}(J\bc_\lambda(-\infty))= \left\langle \left(
     \begin{array}{c}
       \pm\sqrt{b/8+\lambda}-\sqrt{2b}/4 \\
       1 \\
     \end{array}
   \right) \right\rangle, \label{eigen JB-inf}
\end{equation}
\begin{equation}
V^{\pm}(J\bc_\lambda(+\infty))= \left\langle \left(
     \begin{array}{c}
       \pm\sqrt{b/8+\lambda}+\sqrt{2b}/4 \\
       1 \\
     \end{array}
   \right) \right\rangle. \label{eigen JB+inf}
\end{equation}
Moreover,
\begin{equation} \label{homanderindexhatb}
s(\vd, \vn;  V^+(J\bc_\lambda(-\infty)), V^+(J\bc_\lambda(\infty)) )=\begin{cases}
1, \quad if \quad  \lambda\in(-b/8, 0), \\
0, \quad if \quad  \lambda\in(0, +\infty).
\end{cases}
\end{equation}
\end{lem}

\begin{proof} The eigenvalues and eigenspaces follow from direct computations.
To compute the H\"{o}rmander index. Set
$$ \Lambda_\lambda(t)=\left\langle \left(
     \begin{array}{c}
       \sqrt{b/8+\lambda}+t\sqrt{2b}/4 \\
       1 \\
     \end{array}
   \right) \right\rangle, \;\; \text{ for } t\in [-1,1].$$
Then $\Lambda_\lambda(-1)=V^+(J\bc_\lambda(-\infty))$ and $\Lambda_\lambda(1)=V^+(J\bc_\lambda(\infty))$. When $\lambda>0$,
$$ \sqrt{b/8+\lambda}+t\sqrt{2b}/4>0, \;\; \forall t\in[-1,1].$$
Therefore $\Lambda_\lambda(t) $ is transversal to both $\vd$ and $\vn$, for any $t \in [-1, 1]$. Then $\mu(\vd, \Lambda_\lambda)=\mu(\vd, \Lambda_\lambda)=0$. This implies
$$s(\vd, \vn,  V^+(J\bc_\lambda(-\infty)), V^+(J\bc_\lambda(\infty)) )=0.$$
When $-b/8<\lambda<0$,  $\Lambda_\lambda(t)\pitchfork\vd$, $\forall t\in[-1,1]$, which means $\mu(\vd, \Lambda_\lambda)=0$. Meanwhile there is only one negative crossing for $\vn$, so $\mu(\vd, \Lambda_\lambda)=-1$.
As a result,
 $$s(\vd, \vn,  V^+(J\bc_\lambda(-\infty)), V^+(J\bc_\lambda(\infty)) )=1.$$
This completes our proof.
\end{proof}

In the following, let $V^s_{\lmd}(\tau)$ and $V^u_{\lmd}(\tau)$ be the stable and unstable subspaces of the linear Hamiltonian system  \eqref{eq: B hat lmd} according to Definition \ref{dfn: stable subspace}. For $\lmd=0$, $\bc_{\lmd}$ is degenerate according to Definition \ref{dfn: maslov index hete} as will be shown in the next lemma.
\begin{lem}\label{lem:mu=0}  $\mu(\bc_0; \rr)=0$ and $\nu(\bc_0)=1.$
\end{lem}
\begin{proof}
 Notice that $\xi_1(\tau)=r^{\frac{1}{4}}(\tau)(0,1)^T$, $\tau \in \rr$, is a  solution of \eqref{eq: B hat lmd} with $\lmd=0$, where $r(\tau)$ satisfy \eqref{eq: McGehee 2}. Since $q(t)$ is a homothetic solution with negative energy, $r(\tau)$ as well as $\xi_1(\tau)$ goes to $0$, as $\tau\to\pm\infty$. Therefore $V^{s}_0(\tau) = V^{u}_0(\tau)=\langle\xi_{1}(\tau) \rangle$. By Definition \ref{dfn: maslov index hete}, this implies $\nu(\hat{B}_0)=1$. Meanwhile since $V^u_0(\tau)$ has no crossing with $\vd$, we have $\mu(\bc_0; \rr)=0$.
 \end{proof}

Next lemma shows $\bc_{\lmd}$ is non-degenerate, for $\lmd \ne 0$.

\begin{lem}\label{lem:muneq0} $\nu(\bc_\lambda)=0$, for any $\lambda\in (-b/8,0)\cup(0,+\infty). $
\end{lem}
\begin{proof} By a contradiction argument, suppose that there is a $\lambda_0\in(-b/8,0)\cup(0,+\infty)$, such that $\nu(\bc_{\lambda_0})=1$. This means $V^s_{\lmd_0}(0)=V^u_{\lmd_0}(0)$. This further implies $V^s_{\lmd_0}(\tau)=V^u_{\lmd_0}(\tau)$. Then there exists a solution $\xi_1(\tau) = (y(\tau),x(\tau))^T$ of \eqref{eq: B hat lmd}, such that
 $$
\langle \xi_0(\tau) \rangle = V^s_{\lmd_0}(\tau)=V^u_{\lmd_0}(\tau), \;\; \forall \tau \in \rr. $$
Then  $\lim_{\tau \to \pm \infty} \xi_0(\tau) = \lim_{\tau\to \pm\infty}(y(\tau),x(\tau))^T=(0,0)^T$.
Meanwhile $\langle \xi_0(\tau) \rangle \to V^+(J\hat{B}_{\lambda_0}(-\infty))$, when $\tau\to -\infty$. Since $\lambda_0\neq0$, by \eqref{eigen JB-inf}, $y(\tau)$ can not be zero, when $\tau$ is close enough to $-\infty$.

On the other hand, as $\xi_0(\tau)$ satisfies \eqref{eq: B hat lmd}, we have
\begin{equation} \label{eq: xy}
  \begin{aligned}
   y'&= -\frac{v}{4}y+\lambda_0 x,  \\
   x'&= y+\frac{v}{4}x.
  \end{aligned}
  \end{equation}
Differentiate both sides of equation \eqref{eq: xy} with respect to $\tau$
\begin{equation}
 y''=f(\tau)y , \label{yddot1}, \; \text{ where } f(\tau)=\frac{b}{8}(1- \tanh^2(\frac{\sqrt{2b}\tau}{2}))+\lambda_0+\frac{b}{8}>0 .
 \end{equation}
From property (c) of Lemma \ref{lem:nondegen}, there is no non-trivial bounded solution of  \eqref{yddot1}, which is a contradiction to what we obtained above.
\end{proof}

By the previous two lemmas and the decomposition property
$$\nu(\hat{B}_{2})=\#\{\lmd_{i}:\lmd_{i}=0,1\leq i\leq n^{*}-1\}=\dim (\ker(M^{-1}D^2U|_{\mathcal{E}}(s_0))).$$
This immediately implies the following corollary.
\begin{cor}
$\nu(\hat{B})=\nu(\hat{B}_{1})+\nu(\hat{B}_{2})=\dim (\ker(M^{-1}D^2U|_{\mathcal{E}}(s_0))).$
\end{cor}

\begin{lem} \label{lem:neumindex} For $\lambda>-b/8$ and $\lambda\neq0$, $ \mu(\vn, V^u_{\lmd}; \mathbb{R}) =0.$
\end{lem}
\begin{proof} The proof is similar to Lemma \ref{lem:muneq0}. Assume  $V^u_\lmd(\tau) =\langle \xi(\tau) \rangle= \langle (y(\tau),x(\tau))^T \rangle$, for all $\tau$. Then $y(\tau)$  satisfies \eqref{yddot1} and
$ \lim_{\tau\to -\infty}y(\tau) =0$.
Since $y(\tau)$ can not be zero for all $\tau$, from property (b) of Lemma \ref{lem:nondegen}, we have   $y(\tau)\neq0$ for $\tau\in\mathbb{R}$.  This implies that $V^u(\tau)\pitchfork \vn$ for any $\tau\in\mathbb{R}$, so we get the results.

\end{proof}

\begin{prop}  \label{prop: indexhatb}

\begin{equation*}
\mu(\vd, V^u_{\lmd}(\tau); \mathbb{R})=\begin{cases}
1, \quad \text{if} \quad  \lambda\in(-b/8, 0), \\
0, \quad \text{if} \quad  \lambda\in[0, +\infty).
\end{cases}
\end{equation*}
\end{prop}

\begin{proof} For $\lambda=0$, the desired result follows from Lemma \ref{lem:mu=0}.
 For $\lambda\neq0$,  the system is non-degenerate, so
$$  \lim_{\tau\to-\infty}V^u_{\lmd}(\tau)=V^+(J\bc_\lambda(-\infty)),\quad \lim_{\tau\to+\infty}V^u_{\lmd}(\tau)=V^+(J\bc_\lambda(\infty)). $$
Then
$$ \mu(\vd, V^u_{\lmd}; \mathbb{R})-\mu(\vn, V^u_{\lmd}; \mathbb{R})=s(\vd, \vn; V^+(J\bc_\lambda(-\infty)), V^+(J\bc_\lambda(\infty)).$$
Now the desired result follows from \eqref{homanderindexhatb} and Lemma \ref{lem:neumindex}.
\end{proof}

With the above result, we have the following proposition which implies property (b) in Theorem \ref{thm 2 homo}, when $H_0 <0$.

\begin{prop} \label{prop: Mors index NegEng}
 If the energy $H_0$ of $q(t)$ is negative and the normalized central configuration $s_0$ associated with $q(t)$ satisfies the strict non-spiral condition, then
 $$  m^-(q; T^-, T^+) =m^-(M^{-1}D^2U|_{\mathcal{E}}(s_0)). $$
\end{prop}

\begin{proof}
From  \eqref{Vuinfty} and \eqref{eq: MasSum B2},
$$ \mu(\bh; \rr) = \mu(\bh_1;\rr) + \mu(\bh_2; \rr) = \mu(\bh_1;\rr) + \sum_{i=1}^{n^*} \mu(\bc_{\lmd_i}; \rr). $$
Then Proposition \ref{prop:B1} and \ref{prop: indexhatb} imply
 $$  \mu(\hat{B};\rr)=m^-(M^{-1}D^2U|_{\mathcal{E}}(s_0)). $$
This means $\mu(\hat{B};(-\infty, \tau_0])=m^-(M^{-1}D^2U|_{\mathcal{E}}(s_0))$, for any $\tau_0$ large enough. Let $T_0=t(\tau_0)$ be the Newtonian time corresponding to $\tau_0$. Notice that when $\tau_0$ is large enough, $\vd \pitchfork V^u(\tau)$. By Property (a) in Theorem \ref{thm: hyper},
$$  m^-(q; T^-, T_0)=\mu(\hat{B}; (-\infty, \tau_0]) =m^-(M^{-1}D^2U|_{\mathcal{E}}(s_0)).$$
Since $T_0=t(\tau_0)$ goes to $T^+$, when $\tau_0$ goes to $+\infty$. We get
$$  m^-(q; T^-, T^+)=\mu(\hat{B}; \rr) =m^-(M^{-1}D^2U|_{\mathcal{E}}(s_0)).$$ \end{proof}

Now we are ready to prove the last result, which implies Theorem \ref{thm 2 homo}.
\begin{prop} \label{prop: Mors index nonspirial}
 If $s_0$ satisfies the non-spiral condition, then
\begin{equation*}
m^-(q; T^-, T^+)= \begin{cases}
m^-(M^{-1}D^2U|_{\mathcal{E}}(s_0)), & \text{ when } H_0 < 0;\\
0, & \text{ when } H_0 \ge 0.
\end{cases}
\end{equation*}
\end{prop}

\begin{proof}
We shall use a perturbation argument. For $\vep >0$ small enough, set
\bea
\hat{B}(\tau,\varepsilon)=\hat{B}_{1}(\tau)\diamond\hat{B}_{2}(\tau,\varepsilon), \; \text{ where } \hat{B}_{2}(\tau,\varepsilon)=\left(\begin{array}{cc}\hat{M}^{-1} & \frac{1}{4}vI\\
                             \frac{1}{4}vI&-\U_{xx}(x_0)-\varepsilon\hat{M}\end{array}\right).\nonumber
\eea
Then $\hat{B}(\tau)-\hat{B}(\tau,\varepsilon)$ is a non-negative matrix.
By the monotone property of Maslov index,
\bea
\mu(\vd,\hat{\ga}(\tau,\tau_{1})\vd;[\tau_{1},\tau_{2}])\geq
\mu(\vd,\hat{\ga}_{\vep}(\tau,\tau_{1})\vd;[\tau_{1},\tau_{2}]),\label{96}
\eea
where $\hat{\ga}_{\vep}(\tau,\tau_{1})$ satisfies
\begin{equation*}
\begin{cases}
\hat{\ga}'_{\vep}(\tau,\tau_{1}) & =J\hat{B}(\tau,\varepsilon)\hat{\ga}_{\vep}(\tau,\tau_{1}); \\
\hat{\ga}_{\vep}(\tau_{1},\tau_{1}) & =I.
\end{cases}
\end{equation*}
This system satisfies the strict non-spiral condition.

First let's assume $H_0<0$. By previous results,
\bea
\lim_{\tau_{1}\rightarrow-\infty}\lim_{\tau_{2}\rightarrow+\infty}
\mu(\vd,\hat{\ga}_{\vep}(\tau,\tau_{1})\vd;[\tau_{1},\tau_{2}])-n^{*}
=m^{-}(\U_{xx}(x_0)+\varepsilon\hat{M}).\label{97}
\eea
Meanwhile for $\varepsilon$ small enough,
$$ m^{-}(\U_{xx}(x_0)+\varepsilon\hat{M})= m^-(\hat{U}_{xx}(x_0))=m^-(M^{-1}D^2U|_{\mathcal{E}}(s_0)). $$
By the monotone property (\ref{96}) and the Morse index theorem (\ref{Mor-Mas}),
\bea
m^-(q; T^-, T^+)\geq m^-(M^{-1}D^2U|_{\mathcal{E}}(s_0)).\nonumber
\eea
Meanwhile (\ref{97}) implies that for any $\tau_{1} < \tau_{2}$,
\bea
\mu(\vd,\hat{\ga}_\varepsilon(\tau,\tau_{1})\vd;[\tau_{1},\tau_{2}])-n^{*}
\leq m^-(M^{-1}D^2U|_{\mathcal{E}}(s_0))\label{98}
\eea

By a contradiction argument, let's assume the desired identity does not hold. Then
\bea
m^-(q; T^-, T^+)\geq m^-(M^{-1}D^2U|_{\mathcal{E}}(s_0))+1. \nonumber
\eea
According to the Morse index theorem (\ref{Mor-Mas}) and Lemma \ref{lem: Maslov index t tau},
there exist $\tau_{1} < \tau_{2}$, such that
\bea
\mu(\vd,\hat{\ga}(\tau,\tau_{1})\vd;[\tau_{1},\tau_{2}])-n^{*}\geq m^-(M^{-1}D^2U|_{\mathcal{E}}(s_0))+1.\nonumber
\eea
From the index theorem (see \cite{HS09}), the index $\mu(\vd,\hat{\ga}(\tau,\tau_{1})\vd;[\tau_{1},\tau_{2}])$
is equivalent to the Morse index of the Sturm-Liouville operator of this system. This implies such an index will not decrease by a small perturbation. Hence for $\varepsilon$ small enough, we still have
\bea
\mu(\vd,\hat{\ga}_\varepsilon(\tau,\tau_{1})\vd;[\tau_{1},\tau_{2}])-n^{*}\geq m^-(M^{-1}D^2U|_{\mathcal{E}}(s_0))+1.
\nonumber
\eea
This contradicts (\ref{98}). Hence under the non-spiral condition, we still have
\bea
m^-(q; T^-, T^+)=m^-(M^{-1}D^2U|_{\mathcal{E}}(s_0)).\nonumber
\eea

The proof for $H_0\geq0$ is similar. For any $\tau_1<\tau_2$, we have
\bea
m^-(q_\epsilon; t(\tau_1), t(\tau_2))=\mu(\vd,\hat{\ga}_{\vep}(\tau,\tau_{1})\vd;[\tau_{1},\tau_{2}])-n^{*}
=0,\label{99.1}
\eea
where $m^-(q_\epsilon; t(\tau_1), t(\tau_2)) $ is the Morse index of the $\epsilon$ perturbation system. Since the Morse index is non-decrease under small perturbations, we have
$$
m^-(q; t(\tau_1), t(\tau_2))=0,$$
for any $\tau_1 <\tau_2$.  Let $\tau_1\to -\infty$, $\tau_2\to +\infty$ ($\tau_2\to \tau^+$ in the case of $H_0>0$), we get the desired result.

\end{proof}

\mbox{}

\textbf{Acknowledgments}. We thank the anonymous referee for his or her useful comments. We thank Andrea Venturelli for suggesting the term ``doubly asymptotic'', as well as Alain Chenciner, Rick Moeckel, Richard Montgomery for useful comments, that helped improving the terminology of this paper.

The first  author  thanks useful discussions with Vivina Barutello,  Alessandro Portaluri and Susanna Terracini. The last two authors wishes to thank School of Mathematics, Shandong University for its hospitality, where part of the work was done when they were visitors there.

\bibliographystyle{abbrv}

\bibliography{ref-singular}

\end{document}